\documentclass[12pt, reqno]{amsart}
\usepackage{a4,amsthm}
\usepackage{mathabx}

 \input xy   
 \xyoption{all} 
 
\theoremstyle{plain}
\newtheorem{theorem}{Theorem}[section]
\newtheorem{lem}[theorem]{Lemma}
\newtheorem{prop}[theorem]{Proposition}
\newtheorem{cor}[theorem]{Corollary}
\newtheorem{defi}[theorem]{Definition}

\newtheorem{remark}[theorem]{Remark}
\newtheorem{exa}[theorem]{Example}

 \newtheorem*{rem} {Remark} 

\def\Ad{{\rm Ad}}

\def\grad{{\rm grad}}
\def\conv{{\rm conv}}

\def\a{{\mathfrak{a}}}
\def\b{{\mathfrak{b}}}

\def\s{{\mathfrak{s}}}

\def\m{{\mathfrak{m}}}
\def\n{{\mathfrak{n}}}

\def\p{{\mathfrak{p}}}    
\def\k{{\mathfrak{k}}}

\def\g{{\mathfrak{g}}}
\def\l{{\mathfrak{l}}}
\def\h{{\mathfrak{h}}}

\def\B{{\mathbb{B}}}
\def\C{{\mathbb{C}}}
\def\R{{\mathbb{R}}}
\def\HH{{\mathbb{H}}}
\def\N{{\mathbb{N}}}

\def\nsmallskip{\smallskip\noindent}
\def\bbigskip{\bigskip\bigskip}

\def\nmedskip{\medskip\noindent}

\def\buildunder#1#2{\mathrel{\mathop{\kern0pt #2}
\limits_{#1}}}

\def\dds{\frac{d}{ds}{\big |_{s=0}}}
\def\ddt{\frac{d}{dt}{\big |_{t=0}}}
\def\ddx{{\frac{d}{dx}}}
\def\ddy{{\frac{d}{dy}}}

\def\pn{\par\noindent}
\def\sn{\smallskip\noindent}
\def\mn{\medskip\noindent}
\def\bn{\bigskip\noindent}

\def\VD{{\scriptstyle V_D}}
\def\REM #1{}

\begin{document}

\title[Unipotent geometry]{Geometry  of   Hermitian symmetric spaces under the action of
a maximal unipotent group}

\bbigskip

\author[Laura Geatti]{Laura Geatti}
\author[Andrea Iannuzzi]{Andrea Iannuzzi}

\address{Dipartimento di Matematica,
Universit\`a di Roma  ``Tor Vergata", Via della Ricerca Scientifica 1,
I-00133 Roma, Italy} 
\email{geatti@mat.uniroma2.it, iannuzzi@mat.uniroma2.it}

\address{Roma, 26 ottobre 2022}



\thanks {\ \  The authors acknowledge the MIUR Excellence Department Project awarded to the Department of Mathematics, University of Rome ``Tor Vergata", CUP
E83C18000100006.
This research was partially supported by GNSAGA-INDAM.
}

\begin{abstract}
Let $\,G/K\,$ be a non-compact irreducible Hermitian symmetric space of rank $\,r\,$ and let  $\,NAK\,$ be an Iwasawa decomposition  of $\,G$.
By the polydisc theorem,  $\,AK/K\,$ can be regarded as the base
of an $\,r$-dimensional  tube domain  holomorphically embedded in $\,G/K$. As 
every $\,N$-orbit in $\,G/K\,$ intersects $\,AK/K$ in a single point,  there is a one-to-one correspondence  between $\,N$-invariant domains in $\,G/K\,$ and tube domains
in the product of $\,r\,$ copies of the upper half-plane in $\,\C$. 
In this setting we prove a generalization of Bochner's tube theorem. Namely,
an $\,N$-invariant  domain $\,D\,$ in $\,G/K\,$ is Stein if and only if the base $\,\Omega\,$ of the associated tube domain is convex and
``cone invariant". We also obtain a precise description of the envelope of holomorphy of an arbitrary holomorphically separable $\,N$-invariant  domain over $\,G/K$.

An important ingredient for the above results is the characterization of several classes of $\,N$-invariant plurisubharmonic funtions on $\,D\,$ in terms of the corresponding classes of convex functions on $\,\Omega$. This
also leads to an explicit 
 Lie group theoretical description of all $\,N$-invariant potentials of the  Killing  metric on $\,G/K$.
\end{abstract}

\maketitle


\section{Introduction}
\label{Introduction}

The classical Bochner's tube theorem states that the envelope of holomorphy of a tube domain $\,\R^n +  i\Omega\,$ in $\,\C^n\,$ is univalent and coincides  with the convex envelope $\,\R^n+i\,\conv(\Omega)$. 
Moreover,  there is a one-to-one correspondence between the class of $\,\R^n$-invariant plurisubharmonic functions on a  Stein tube domain in    $\,\C^n\,$   and the class of convex functions on its base   in $\,\R^n$ (cf.\,\cite{Gun90}).

Here our goal is to  obtain analogous  results    in the setting of an irreducible  Hermitian symmetric space  of the  non-compact type,  under  the action of  a maximal unipotent group of holomorphic automorphisms. 

Any such space can be realized as a quotient   $G/K$,  where  $G$ is a non-compact real simple Lie group and $K$ is a maximal compact subgroup of $G$. 
Let $\,\g=\n\oplus \a \oplus \k\,$ be an Iwasawa decomposition of $\,\g$, where $\,\n\,$ is a maximal nilpotent subalgebra, $\,\a\,$ is a maximally split abelian subalgebra and $\,\k\,$ is the Lie algebra of $\,K$. 
The integer $\,r:=\dim\a\,$ is by definition the rank of~$\,G/K$.

Let $\,NAK\,$ be the corresponding  Iwasawa decomposition of $\,G\,$,  where  $\,A:=\exp\a\,$   and $\,N:=\exp \n$.   
The group $\,N\,$ acts on $\,G/K\,$  by biholomorphisms  and   every $\,N$-orbit  in $\,G/K\,$ intersects the smooth, real
$\,r$-dimensional  submanifold  $\,A\cdot eK\,$ transversally in a single point.

 As the space $\,G/K\,$ is  {\it Hermitian} symmetric, $\,G\,$ contains $\,r\,$  pairwise commuting subgroups isomorphic to  $\,SL(2,\R)$. 
  The  orbit of the base point $\,eK\in G/K\,$ under the product of such subgroups is  a  closed complex submanifold of $\,G/K$
 which contains $\,A\cdot eK\,$ and is  biholomorphic to $\,\HH^r$, the product  of $\,r$ copies of the upper half-plane in $\, \C$.    Moreover,
   every $N$-orbit in $G/K$ intersects $\,\HH^r\,$ in an $\,\R^r$-orbit. 
 
 This fact   is an analogue of the polydisk theorem  and determines a one-to-one correspondence between  $N$-invariant domains   in $G/K$ and tube domains  in $\,\HH^r\,$ (cf. Prop.\,\ref{FACT1} and Cor.\,\ref{ASSOCIATEDTUBE}).  If $\,D\,$ is an $\,N$-invariant domain in $\,G/K$, then it is in terms of the base  $\,\Omega\,$  of  the associated tube domain in $\,\HH^r\,$ that the properties of   $\,N$-invariant objects on $\,D\,$ can  be best described.
 
 Define the cone 
 $$   \textstyle C:=\begin{cases} (\R^{>0})^r,  \hbox{ in the non-tube case,}\\ (\R^{>0})^{r-1}\times  \{0\}, \hbox{ in the  tube case.}\end{cases} $$ 
A set $\,\Omega \subset \R^r\,$ is  $\,C$-invariant if $\,{\bf y}\in\Omega\,$ implies $\,{\bf y}+{\bf v}\in\Omega$, for all~$\,{\bf v}\in C$. Our generalizion of Bochner's  tube thorem is as follows

\nmedskip
{\bf Theorem \ref{CASOLISCIO1}.}
{\sl Let $\,G/K\,$ be a non-compact irreducible Hermitian symmetric space of rank $r$. Let $\,D\,$ be an $\,N$-invariant domain  in $\,G/K\,$
  and let $\,\R^r+i\Omega\,$ be the associated $\,r$-dimensional tube domain.
Then  $\,D\,$   is Stein if and only if 
 $\,\Omega\,$ is convex and $\,C$-invariant.}

\medskip
We also show that a 
 holomorphically separable, $\,N$-equivariant,  Riemann domain  over $\,G/K\,$ is necessarily  univalent
(cf.\,Prop.$\,$\ref{SCHLICHT}).  This  implies  the following corollary.  

 \nmedskip
{\bf Corollary \ref{ENVELOPE}.} {\sl The envelope of holomorphy $\,\widehat D\,$ of an
$\,N$-invariant domain $\,D\,$ in $\,G/K\,$ is the smallest Stein domain in $\,G/K\,$ containing $\,D$. The base 
$\,\widehat \Omega\,$ of the $\,r$-dimensional   tube domain associated to  $\,\widehat D\,$  is  the convex, $\,C$-invariant hull of~$\,\Omega$.}

\medskip
One  approach to the proof of the above theorem uses smooth $\,N$-invariant functions. There is a one-to-one correspondence between  $\,N$-invariant functions on $D$  and  functions on $\,\Omega$,   and such correspondence preserves regularity.   
An important ingredient  is   the  computation of the Levi form of a smooth  $\,N$-invariant function  $\,f\colon D\to \R\,$  in terms of the Hessian and the gradient of the corresponding function $\,\widehat f\colon \Omega\to \R$. To this end, a simple pluripotential argument enables us to 
exploit the restricted root decomposition of $\,\n$ (cf.\,Prop.\,\ref{LEVI} and Prop.\,\ref{PSHPOLIDISC2}).

Then, in the smooth case, the proof  of Theorem\,\ref{CASOLISCIO1} is carried out by showing that $\,D\,$ is Levi pseudoconvex, and therefore Stein,  
 if and only if the base $\,\Omega\,$ of the associated tube domain is convex and $C$-invariant.

The general case follows  from the smooth case   by exhausting  $\,D\,$ with  an increasing sequence  of Stein, $\,N$-invariant domains with smooth boundary. For this we adapt a classical approximation method for convex functions on convex domains to our $\,C$-invariant  context.

In Section 6, an alternative  proof  of   Theorem\,\ref{CASOLISCIO1}    is carried out by realizing $\,G/K\,$ as a Siegel domain and   by combining some results from the theory of normal $\,J$-algebras   with  some convexity arguments. 
 
 \smallskip
The aformentioned computation of the Levi form leads to a characterization of smooth $N$-invariant plurisubharmonic functions on $\,N$-invariant domains in $G/K$  in terms of the corresponding functions on $\,\Omega$.   
By classical approximation methods, a similar characterization  is obtained for   arbitrary $\,N$-invariant  (strictly) plurisubharmonic functions on $\,D$. 
 In order to formulate such results we need the following definition. 

Let  $\,\widehat f\colon \Omega\to \R\,$ be a function defined on 
 a $\,C$-invariant domain in $\,(\R^{>0})^r\,$
 and  let $\,\overline C\,$ be the closure of the cone $\,C$.  Then $\,\widehat f\,$ is $\,\overline C$-decreasing if for every $\,{\bf y} \in \Omega\,$ and
$\,{\bf v} \in \overline C
$ the restriction of $\,\widehat f\,$ to the half-line $\,\{ {\bf y} + t {\bf v } \ : \ t \geq 0 \}\,$ is decreasing.

\bn
{\bf Theorem.} (see Thm.\,\ref{BIJECTIVEK})
{\sl Let $\,D\,$ be a Stein, $\,N$-invariant domain in a non-compact, irreducible 
Hermitian symmetric space $\,G/K\,$ of rank $\,r\,$ and let $\, \Omega\,$ be the base of the associated $\,r$-dimensional tube domain.

\smallskip
 An $\,N$-invariant  function   $\,f\colon D\to \R\,$ is (strictly)  plurisubharmonic  if and only if the corresponding function 
 $\,\widehat f\colon \Omega\to\R \,$ is (stably) convex and   $\,\overline C$-decreasing.
 
 \sn
It follows  that  every $\,N$-invariant plurisubharmonic function  on $\,D\,$ is continuous.}

\medskip
In fact, the above theorem holds true both in the smooth and non-smooth context, and  can be regarded as a generalization of  the well known
result for $\,\R^n$-invariant plurisubharmonic functions on tube domains in $\,\C^n$  
(see Sect.\,5 for precise definitions and statements). 

In the appendix, as an   application of   our methods  we  explicitly  determine all the $\,N$-invariant potentials
of the Killing metric on $\,G/K\,$ in a Lie group theoretical fashion.


\section{Preliminaries}
\label{PRELIMINARIES}

\medskip

Let $\,G/K\,$ be an irreducible Hermitian symmetric space, where $\,G\,$ is a real non-compact semisimple Lie group and $\,K\,$ is a maximal compact subgroup of $\,G$.  Let $\,\g\,$   and   $\,\k\,$ be the respective Lie lagebras.  Let $\,\g=\k\,\oplus\,\p$ be the Cartan decomposition of $\,\g\,$ with respect to $\,\k$, with  Cartan involution $\,\theta$. Denote by $\,B(\,\cdot\,,\,\cdot\,)\,$ both the Killing form of $\,\g\,$ and  its $\,\C$-linear extension to $\,\g^\C\,$ (which coincides with the Killing form of $\,\g^\C$). 

Let  $\,\a\,$ be a maximal abelian subspace in $\,\p$.  The dimension  of $\,\a\,$ is by definition the {\it rank} $r$ of
$\,G/K$.  Let 
 $\g= \m\oplus \a\oplus\bigoplus_{\alpha\in\Sigma}\g^\alpha $  be the restricted root decomposition of $\g$ determined by the adjoint action of $\a$, 
where $\,\m\,$ denotes the centralizer of $\,\a\,$ in $\,\k$.
For a simple Lie algebra   of Hermitian type  $\,\g\,$,  the restricted root system is either of type 
$\,C_r\,$ (if $\,G/K\,$ is of tube type) or of type $\,BC_r\,$ (if $\,G/K\,$ is not of tube type), i.e.  there exists  a basis
$\,\{e_1,\ldots,e_r\}\,$ of $\,\a^*\,$  for which  a positive system $\Sigma^+$ is given by 
 $$\Sigma^+=\{2e_j, ~1\le j\le r,~~e_k\pm e_l,~ 1\le k< l\le r\},\quad \hbox{ for type $\,C_r$}, $$
 $$\Sigma^+=\{e_j,~2e_j,~1\le j\le r,~~e_k\pm e_l,~~1\le k<l\le r\},\quad \hbox{ for type $\,BC_r$}\,.$$
The roots
 $ 2e_1,\,\dots\,,2e_r \,$ 
form a maximal set of  long strongly orthogonal positive restricted  roots.
The root spaces   $\, \g^{2e_1},\ldots,\g^{2e_r}\,$ are one-dimensional and one can choose generators $\,\ E^j \in \g^{2e_j}\,$ such  that 
 the $\,\s \l (2)$-triples
$\,\{E^j,~\theta  E^j,~ A_j:=[\theta  E^j,\, E^j]\}\,$
are normalized as follows  
\begin{equation}\label{NORMALIZ1}
\textstyle [ A_j,\, E^l]=\delta_{jl}2  E^l, \quad \hbox{for}\quad j,l=1,\ldots,r. \end{equation}
Denote by $\,I_0\,$ the $\,G$-invariant  complex structure  of $\,G/K$.
 We   assume that~$I_0(E^j-\theta E^j) = A_j$.
By the strong orthogonality of $ 2e_1,\ldots, 2e_r$,  the  vectors $\, A_1,\ldots, A_r \,$ form  a $\,B$-orthogonal basis of $\,\a\,$, dual to  $e_1,\ldots,e_r$ of $\a^*$,  and 
the  associated $\,\s \l (2)$-triples pairwise commute.  

Let $\g=\n\oplus \a\oplus\k$ be the  Iwasawa decomposition subordinated to $\Sigma^+$, where  $\n=\oplus_{\alpha\in\Sigma^+}\g^\alpha$,  and let $G=NAK$ be the corresponding Iwasawa decomposition of $G$. 
Then $S=NA$ is a real split solvable group  acting freely and  transitively on $G/K$. 
 In particular, the tangent space to $G/K$ at the base point $eK$ can be identified with the Lie algebra $\s=\n\oplus\a$.

The map $\phi\colon \s\to \p$, given  by 
 $\phi(X):={1\over 2}(X-\theta X),$  is an isomorphism of vector  spaces.  
As a consequence, 
\begin{equation}\label{INNERPRODUCT}\textstyle  \langle X\,,Y \rangle :=B(\phi(X)\,,\phi(Y))=-{1\over 2} B(X\,,\theta Y),  \end{equation}
for $X,\,Y\in\s$, defines a positive definite symmetric bilinear form on $\s$.
Moreover, the map $J\colon \s\to\s$, given by
\begin{equation}\label{COMPLEXJ} 
\textstyle JX:=\phi^{-1}\circ I_0\circ \phi(X),\end{equation}
defines a complex structure on $\s$, such that $\phi(JX)=I_0\phi(X)$.
The complex structure $\,J\,$ permutes the restricted root spaces 
of~$\,\s$ (cf. \cite{RoVe73}), namely 
\begin{equation}
\label{CPLXBIS}
\,J\a=\bigoplus_{j=1}^r \g^{2e_j}, \quad J\g^{e_j-e_l}=\g^{e_j+e_l},  \quad J\g^{e_j}=\g^{e_j}\,.
 \end{equation}

In order to  obtain  a precise description of~$J$ on~$\s$, we  recall a few more  facts.  Let $ \g^\C=\h^\C\oplus \bigoplus_{\mu\in\Delta} \g^\mu$ be the root decomposition of $\g^\C$ with respect
to  a maximally split Cartan subalgebra $\h=\b\oplus \a$ of $\g$, where $\b$ is an abelian subalgebra of~$\m$. Let $\sigma$ be the conjugation
of $\g^\C$ with respect to~$\g$. Let $\theta$ denote also the $\C$-linear extension of 
$\theta$ to $\g^\C$. One has $\theta\sigma=\sigma\theta$. Write $\overline Z:=\sigma Z$, for $Z\in\g^\C$. As  $\sigma$ and $\theta$ stabilize  $\h$, they induce   actions on $\Delta$, defined by
$\bar\mu (H):=\overline{\mu(H)}$ and $\theta\mu(H):=\mu(\theta(H))$, for $H\in\h$, respectively.  Fix a positive root system  $\Delta^+$ compatible with 
$\Sigma^+$, meaning that $\mu|_\a =Re(\mu) \in \Sigma^+$ implies $\mu\in\Delta^+$. Then~$\sigma\Delta^+= \Delta^+$.

Given a restricted root $\alpha\in\Sigma $, the corresponding  restricted root space $\g^\alpha$ decomposes into the direct sum of ordinary root spaces with respect to the  Cartan subalgebra $\h$ as follows
$$\g^\alpha=\Big( \bigoplus_{\mu\in \Delta,\, \mu\not=\bar\mu\atop Re(\mu)=\alpha} \g^\mu\oplus \g^{\bar\mu}\quad \oplus \g^\lambda
\Big) \cap \g,$$
where  $\lambda \in \Delta$ is possibly a  root  satisfying $\lambda=\bar\lambda=\alpha$. 
The next lemma is obtained by combining Lemma 2.2 in {\rm \cite{GeIa21}} with (\ref{COMPLEXJ}).  
 
\begin{lem} \label{COMPLEXSTRUCTURE} {\bf (the complex structure $J$ on $\s$)}.
\item{$(a)$}  For $j=1,\ldots,r$, let $A_j\in\a$ and $E^j\in\g^{2e_j}$ be elements normalized as in $($\ref{NORMALIZ1}$)$. 
Then $JE^j ={1\over 2}A_j$ and $JA_j=-2 E^j$.  

\sn
\item{$(b)$}  Let   $X=Z^\mu+\overline{Z^\mu}\in \g^{e_j-e_l}$, 
 where $\mu\in\Delta^+$ is a root satisfying 
  $Re(\mu)=e_j-e_l$ and    $Z^\mu\in\g^\mu$ (if $\bar\mu=\mu$, we may assume $Z^\mu=\overline{Z^\mu}$
  and set $X=Z^\mu$).
Then $JX = [ E^l,X]\in \g^{e_j+e_l}$. \\
Let   $X=Z^\mu+\overline{Z^\mu}\in \g^{e_j+e_l}$, 
 where $\mu\in\Delta^+$ is a root satisfying 
  $Re(\mu)=e_j+e_l$ and    $Z^\mu\in\g^\mu$ (if $\bar\mu=\mu$, we may assume $Z^\mu=\overline{Z^\mu}$
  and set $X=Z^\mu$).
Then $JX = [ \theta E^l,X]\in \g^{e_j-e_l}$. 

\sn
\item{$(c)$} Let $ X  =Z^\mu+\overline{Z^\mu}\in\g^{e_j}$, where  $\mu$ is a root in
$\Delta^+$ satisfying  $Re(\mu)=e_j$ and   $Z^\mu\in\g^\mu$ (as $\dim\g^{e_j}$ is even, one necessarily has 
$\bar\mu\not= \mu)$. Then $JX =iZ^\mu + \overline{iZ^\mu}\in\g^{e_j}$.
\end{lem}

\begin{remark} 
\label{BASIS2} {\bf (a $J$-stable basis of $\s$)} In view of Lemma \ref{COMPLEXSTRUCTURE}, one can choose a $J$-stable basis of $\s$, compatible with the restricted root decomposition. 

\sn
$ (a)$ As a  basis of   $\a\oplus J\a$,  take pairs of   elements 
$A_j$, $ JA_j =-2E^j$, for $j=1,\ldots,r$, normalized as in $($\ref{NORMALIZ1}$)$. 

\sn
$ (b)$ As a basis of  $\g^{e_j-e_l}\oplus\,\g^{e_j+e_l}$,   
take 4-tuples of elements  
\begin{equation}
\label{BASEDET} X=Z^\mu+\overline{Z^\mu},\quad X'=iZ^\mu+\overline{iZ^\mu},\quad
JX=[E^l,X],\quad JX'= [E^l,X'], \end{equation}
parametrized by the pairs of roots $\mu\not=\bar\mu \in\Delta^+$ satisfying   $Re(\mu)=e_j-e_l $ $($with no repetition$)$, 
with  $Z^\mu $  a root vector in $\g^\mu$. 
For $\mu= \bar \mu$, one may assume $Z^\mu=\overline{Z^\mu}$ and take the pair $X=Z^\mu,\,  JX=[E^l,X]$.  \pn

\sn
$ (c)$ As a  basis of $\g^{e_j}$ $($non-tube case$)$,  take pairs of elements $$X=Z^\mu+\overline{Z^\mu}, \quad JX=  iZ^\mu+\overline{iZ^\mu},$$ parametrized by the pairs of roots $\mu\not=\bar\mu \in\Delta^+$ satisfying  $Re(\mu)=e_j $ $($with no repetition$)$, 
 with  $Z^\mu\in \g^\mu$.
\end{remark}

\smallskip
The next lemma contains some identities which are needed in Section 3. Its proof is essentially contained in 
\cite{GeIa21}, Lemma\,2.4. 

\begin{lem} 
\label{NBRACKETS}
 Let  $\,\mu\in \Delta^+\,$ be a root satisfying $\,Re(\mu)=e_j-e_l\,$ and  let $\,Z^\mu\,$ a root vector in
 $\,\g^\mu$. 
 Let  $\,X=Z^\mu+\overline Z^\mu  \in \g^{e_j-e_l}\,$ and $\,JX=[E^l,X]\in\g^{e_j+e_l}$.
 If $\, \overline \mu \not= \mu$, let  $\,X'=iZ^\mu+\overline {iZ^\mu}\,$ and $\,JX'=[E^l,X']$. 
 Then
 
\mn
$(a)$ $\,[JX,X] = [JX',X']=s E^j $, for some $\,s  \in \R,\,s\not=0$;

\mn
$(b)$ $[JX',  X]  =0.$ 
 
\mn
Let  $\,\mu\,$ be a root in $\,\Delta^+$, with $\,Re(\mu)=e_j\,$ $($non-tube case$)$ and let  $\,Z^\mu\,$
 be a root vector in $\,\g^\mu$.
Let  $\,X=Z^\mu+\overline Z^\mu \,$ and $\,JX=iZ^\mu+\overline{iZ^\mu}$. Then 

\mn
 $(c)$ $\,[JX,X] =tE^j$, for some $\,t  \in \R,\,t\not=0$. 
 
\end{lem}

\REM{
\begin{proof}   
\pn
(a)  
One has $[JX,X]=[JX',X']= 2Re[[ E^l,Z^\mu],\overline{Z^\mu}]\in   \g^{ 2e_j}.$ 
Since  the root  space $\g^{2e_j}$ is 1-dimensional, then    
 $$[JX,X] =[JX',X'] =s E^j ,\quad \hbox{ for some $s\in\R$}.$$
\sn
(b) One has  
$$[JX',X]=[[E^l,iZ^\mu+\overline{iZ^\mu}],Z^\mu+\overline{Z^\mu}]$$
$$=i [[E^l,Z^\mu],Z^\mu]+i[[E^l,Z^\mu],\overline{Z^\mu}]-i [[E^l,\overline{Z^\mu}],Z^\mu] -i [[E^l,\overline{Z^\mu}],\overline{Z^\mu}].$$
The first and the fourth terms of the above expression are   both zero because otherwise there would exist a root in $\Delta^+$ with real part equal to $2e_j$ and non-zero imaginary part. 
The second and the third term sum up to zero by the Jacobi identity and the fact that $ [Z^\mu,\overline{Z^\mu}]=0$.
\sn
(c)  It follows directly from the fact that  $\,[JX,  X]\in\g^{2e_j}$.
   \end{proof}         }

\bigskip
\section{The Levi form of an $N$-invariant function on $G/K$}

\medskip

Let $\,G/K\,$ be  a non-compact,
 irreducible  Hermitian symmetric space  of rank $\,r$, and let  
 $\,G=N \exp (\a)\, K\,$ be an   Iwasawa decomposition of $\,G$.    
Let $D$ be an  $\,N$-invariant domain   in 
$\,G/K\,$. Then $\,D\,$  is uniquely determined by a domain $\,{\mathcal D}\,$ in~$\,\a\,$ by 
\begin{equation}\label{DA}\, D:=N \exp (\mathcal D) \cdot e K  \,.\end{equation}
Similarly, an $N$-invariant function $f:D \to \R$ is uniquely determined by the   function 
 $\widetilde f:\mathcal D \to \R$, defined by 
 \begin{equation} \label{EFFETILDE}\widetilde f(H): =f(\exp (H)K).\end{equation}

\sn
The goal of this section is to express the {\it Levi form},  i.e.\,the real symmetric $J$-invariant bilinear form
\begin{equation}\label{HESSIAN} h_f(\,\cdot\,,\,\cdot\,):= -dd^c f(\,\cdot\,,J\,\cdot\,),\end{equation} of a smooth $\,N$-invariant function $f$ on $D$, in terms of the   first and second derivatives of the corresponding function $\widetilde f$ on  $\,\mathcal D$. This will enable us to characterize smooth $N$-invariant strictly plurisubharmonic functions on a Stein $N$-invariant domain $D$ in $G/K$ by  appropriate 
conditions on the corresponding functions on $\mathcal D$ (Prop.\,\ref{LEVI}).
As $f$ is 
$N$-invariant,  $h_f$ is  $N$-invariant as well. Therefore it will be sufficient
to carry out the computation along the slice $\exp (\mathcal D) \cdot eK$, which meets all $\,N$-orbits.

For $\,X \in \g$, denote by 
$\,\widetilde X\,$ the vector field on $\,G/K\,$  induced   by the left $\,G$-action. Its value at  $z\in G/K$ is given by    
  \begin{equation}\label{TILDEFIELDS}\,
  \textstyle\widetilde X_z:=\dds \exp sX \cdot z. \end{equation}
 Let  $X  \in \g^\alpha$, for    $\alpha\in \Sigma^+ \cup \{0\}$ (here  $X\in\a$, when $\alpha=0$).  If $z=aK$, with  $\,a=\exp H$  and $\,H \in \a$, then  the  vector field $\widetilde X$ can also be expressed as   
 \begin{equation} \label{TILDE}\widetilde{X }_z =e^{-\alpha(H)} a_*X. 
 \end{equation}
 Set \begin{equation}\label{BI}{\bf b}:=B(A_1,A_1)=\ldots=B(A_r,A_r),\end{equation}
which is a real positive constant only depending on the Lie algebra $\g$.

\bn

\medskip
\begin{prop}\label{LEVI} 
 Let $\,D\,$ be an $\,N$-invariant domain in $G/K$ and let $\,f:D \to \R\,$ be a smooth $\,N$-invariant function. Fix $\,a=\exp H$, with $H=\sum_ja_jA_j\in \mathcal D$.  Then, in   the basis  of $\,\s\,$ defined in  Remark \ref{BASIS2},    the form  $h_f$  at $z=aK\in D$ is given as follows.\ 
\begin{itemize}
\smallskip
\item[(i)]  The spaces $\,a_*\a$, $\,a_*J\a$, $\,a_* \g^{e_j- e_l}\,$, $\,a_* \g^{e_j+ e_l}\,$ 
and $a_* \g^{e_j}\,$ are pairwise $\,h_f$-orthogonal.
  \end{itemize}
\begin{itemize}
\smallskip
\item[(ii)]  For $A_j,A_l\in\a$ 
one has

\noindent
$\,h_f(a_*A_j,a_*A_l)=-2 \delta_{jl} \frac{\partial \widetilde f}{\partial a_l}(H) +\frac{\partial^2 \widetilde f}{\partial a_j\partial a_l}(H)$.
\end{itemize} 

\sn
On  the blocks $a_*\g^{e_j-e_l}$ and $a_*\g^{2e_j}$ the restriction of $h_f$ is diagonal and the only non-zero entries are given as follows.
\begin{itemize}
\smallskip
\item[(iii)]   For 
$X,\, X' \in \g^{e_j- e_l}$ as in Remark \ref{BASIS2}(b), 
one has

\smallskip
\noindent
$$\textstyle \,h_f(a_*X ,a_*X )= -2 \frac{\Vert X\Vert^2}{{\bf b}} \frac{\partial \widetilde f}{\partial a_j}(H),\qquad h_f(a_*X' ,a_*X' )=-2 \frac{\Vert X'\Vert^2}{{\bf b}} \frac{\partial \widetilde f}{\partial a_j}(H).$$

\smallskip
\item[(iv)] $($non-tube case$)$ For $X\in \g^{e_j}$  as in Remark  \ref{BASIS2}(c), one has
\smallskip
\noindent
$$\textstyle \,h_f(a_*X ,a_*X )=    -2 \frac{\Vert X\Vert^2}{{\bf b}}
\frac{\partial \widetilde f}
{  \partial a_j}(H) .$$ 
   \end{itemize}
   On the remaining blocks $h_f$ is determined by (\ref{CPLXBIS}), the $J$-invariance of $h_f$, (i) and (iii) above.
 \end{prop}

\sn
\begin{proof}  Let $f\colon G/K\to\R$ be a smooth $N$-invariant function.  The computation of   $h_f$ uses the fact  that,   for $X\in\n$, the function $\mu^X\colon G/K\to\R$, given by $\mu^X(z):=d^cf(\widetilde X_z)$, satisfies the identity 
\begin{equation}\label{DIFFER} 
\textstyle d\mu^X=-\iota_{\widetilde X}dd^cf, \end{equation} 
where $d^cf:=df\circ J$ (see \cite{HeSc07}, Lemma 7.1 and \cite{GeIa21}, Sect.\,2).
We begin by determining  $d^cf(\widetilde X_{z})$, for $X\in \n$ and
$z \in G/K$.
By the $\,N$-invariance of $\,f\,$ and of $\,J$ one has 
\begin{equation}\label{DCf} \textstyle  \, d^cf(\widetilde X_{n\cdot z})=d^cf( \widetilde {\Ad_{n^{-1}}X}_{z})\,,\end{equation}

\nsmallskip
for every $\,z \in G/K\,$ and $\,n \in N$. Thus it is sufficient to take $z=aK\in\exp({\mathcal D}) \cdot eK$.
Let   $H=\sum a_jA_j\in\mathcal D$ and $a=\exp H$. Then 
\begin{equation}\label{DCf2}
\textstyle  d^cf(\widetilde X_z)=\begin{cases} 
{1\over 2} e^{-2a_j}\frac{\partial\widetilde f}{\partial a_j}(H)\,, \quad  {\rm for}\ X =E^j\in\g^{2e_j}\\
\quad \quad 0\,, \quad \quad 
\quad \quad \,
{\rm for} \ X\in\g^\alpha,\, {\rm with}\ \alpha \in \Sigma^+ \setminus \{2e_1,\dots,2e_r\}.
\\ \end{cases}\end{equation}
The first part of  equation (\ref{DCf2}) follows from (\ref{TILDE}) and Lemma \ref{COMPLEXSTRUCTURE}\,(a):    
$$\textstyle d^cf((\widetilde {E^j})_z)= 
   e^{-2e_j(H)}df(a_* JE^j)
=\textstyle  {1\over 2} e^{-2a_j} \dds \widetilde f(H+ sA_j)= {1\over 2} e^{-2a_j} \frac{\partial\widetilde f}{\partial a_j}(H).$$
For the second part, let $X\in\g^\alpha$, with $\alpha\in \Sigma^+ \setminus \{2e_1,\dots,2e_r\}$. 
Then $JX\in\g^\beta$, with $\beta\in\Sigma^+$. By (\ref{TILDE})  and the $N$-invariance of $f$, one obtains the desired result
$$ \textstyle  d^cf(\widetilde X_z)=  e^{-\alpha(H)+\beta(H)} df(\widetilde{JX}_z)=0.$$

\mn
{\bf (i) Orthogonality of the blocks.} 
 Let $X \in \g^\alpha$ and  $Y\in  \g^\gamma$, where  $\alpha \in  \Sigma^+$ and  $\gamma \in \{0\} \cup (\Sigma^+ \setminus \{2e_1,\dots,2e_r\})$ are distinct restricted roots (here  $Y\in\a$, when $\gamma=0$).   Then  $JY  \in\g^\beta$, for some  $\beta \in \Sigma^+$.   
 By (\ref{TILDE}) and (\ref{DIFFER}), one has
$$h_f(a_*X,a_*Y)=-dd^cf(a_*X,a_*J Y)=-e^{\alpha(H)+\beta(H)}dd^cf(\widetilde X_z, \widetilde{JY}_z)$$
$$ \textstyle = e^{\alpha(H)+\beta(H)}d\mu^X(\widetilde{JY}_z)= e^{\alpha(H)+\beta(H)}\dds\mu^X(\exp sJY\cdot z)$$
$$ \textstyle = e^{\alpha(H)+\beta(H)}\dds d^cf(\widetilde X_{\exp sJY\cdot z})
= e^{\alpha(H)+\beta(H)}\dds d^cf(\widetilde{Ad_{\exp(-sJY)}X}_z)$$
$$ \textstyle = e^{\alpha(H)+\beta(H)}\dds d^cf(\widetilde X_z-s\widetilde{[JY,X]}_z+o(s^2))$$
\begin{equation} \label{FORMULONE}=-e^{\alpha(H)+\beta(H)}d^cf(\widetilde{[JY,X]}_z).\end{equation}
The brackets  $[JY,X]$ lie  in $\g^{\alpha+\beta}$. Since $\alpha\not=\gamma$, one sees that  $\alpha+\beta \not= 2e_1, \ldots, 2e_r$.  Then,  by  (\ref{DCf2}), the expression (\ref{FORMULONE}) vanishes, proving the orthogonality  of $a_*\g^\alpha$ and 
$a_*\g^\gamma$, for all $\alpha$ and $\gamma$  as above. The $J$-invariance of $h_f$ implies that  $ a_*\a$ is orthogonal to $ a_* \g^\beta$, for all $\beta \in  \Sigma^+ $,  and  concludes the proof of (i). 

\medskip
Next we determine the form  $h_f$ on the essential  blocks.
 
\sn
{\bf (ii) The  form $h_f$ on $a_*\a$.}

\sn
Let $A_j,A_l\in\a$. Since $JA_l=-2E^l$, one has
$$\textstyle h_f(a_*A_j,a_*A_l)=-2dd^cf(a_*E^l,a_*A_j)=-2e^{2e_l(H)} dd^cf((\widetilde{E^l})_z,(\widetilde{A_j}){_z})$$
$$\textstyle =2e^{2e_l(H)} d\mu^{E^l}((\widetilde{A_j}){_z})=2e^{2e_l(H)} \ddt \mu^{E^l}(\exp tA_j \cdot z)$$
$$\textstyle
=2e^{2e_l(H)} \ddt d^cf((\widetilde{E^l})_{\exp tA_j \cdot z}),$$
which, by (\ref{DCf2}), becomes
$$\textstyle =2e^{2e_l(H)} \ddt {1\over 2}e^{-2e_l(H+tA_j)}\frac{\partial\widetilde f}{\partial a_l}(H+tA_j) =
-2 \frac{\partial \widetilde f}{\partial a_l}(H)\delta_{lj}  +\frac{\partial^2 \widetilde f}{\partial a_j\partial a_l}(H).$$
 This concludes the proof of (ii).

\mn
 {\bf \boldmath (iii)   The    form $h_f$ on $a_*\g^{e_j-e_l}$. }

\sn

Let $X,\,X' \in \g^{e_j-e_l} $ be elements of the basis  given in  Remark \ref{BASIS2}\,(b). Then  $JX,\, JX'\in\ \g^{e_j+e_l} $. 
From (\ref{FORMULONE}), (\ref{DCf2}) and Lemma \ref{NBRACKETS}(a)  one has
 $$\textstyle  h_f(a_*X,a_*X )=-dd^cf(a_*X,a_*JX)$$
$$\textstyle = - e^{ (e_j+e_l)(H)} e^{ (e_j-e_l)(H)}d^cf(\widetilde{[JX,X]}_z)$$
\begin{equation} \label{ESSE} \textstyle =- e^{ 2e_j(H)}\left(s d^cf((\widetilde{E^j})_z) \right)
 =-{s \over 2} \frac{\partial \widetilde f}{\partial a_j}(H)\,, \end{equation} 
for some $\,s\in\R\setminus\{0\}$.  By Remark \ref{EXPLANATION}, one has $s>0$.  By  the comparison of (\ref{ESSE})  with the formula obtained in Remark\,\ref{CONSTANTS},  one deduces the  exact value of $s$, namely 
$\textstyle \,s = \frac{4 \Vert X\Vert^2}{{\bf b}}$.   
Therefore, one has
$$
\textstyle h_f(a_*X ,a_*X )=-2\frac{\Vert X \Vert^2}{{\bf b}}\frac{\partial \widetilde f}{\partial a_j}(H), \qquad  h_f(a_*X',a_*X')=-2\frac{\Vert X'\Vert^2}{{\bf b}}\frac{\partial \widetilde f}{\partial a_j}(H),$$   as stated.
From  (\ref{FORMULONE}) and Lemma \ref{NBRACKETS}(b), one obtains   $\,h_f(a_*X,a_*X')=0 $. 
From  (\ref{FORMULONE}), the skew symmetry of $dd^cf$ and the fact that $2(e_j-e_l)\not\in\Sigma^+$, one obtains   $h_f(a_*X,a_*JX )=h_f(a_*X,a_*JX')=0,$  respectively.
Finally, let $X=Z^\mu+ \overline {Z^\mu},\,$ and $\, Y=Z^\nu+ \overline {Z^\nu}$ be elements of the basis of $\g^{e_j-e_l}$  given in Remark \ref{BASIS2}\,(b), for  $\mu,\, \nu\in \Delta^+$  distinct roots satisfying  $\nu\not=\mu,\,\bar \mu$. Then, by   (\ref{FORMULONE})  and Lemma \ref{COMPLEXSTRUCTURE}(b)  
one has
$$h_f(a_*X,a_*Y)  =-e^{2e_j(H)}d^cf(\widetilde{[JY,X]}_z)=0,$$  
since no non-real roots in $\Delta$  have real part equal to $2e_j$. This completes  the proof of (iii).

\mn
 {\bf \boldmath (iv)   The   Hermitian form $h_f$ on $a_*\g^{e_j}$. }

\sn
Let   $X=Z^\mu+\overline{Z^\mu}$ and 
$JX=iZ^\mu + \overline{iZ^\mu}$ be   elements of the basis of $\g^{e_j}$ given in Remark \ref{BASIS2}\,(c). 
Then, from (\ref{FORMULONE}) and Lemma \ref{NBRACKETS}\,(c), one obtains 
$$\textstyle h_f(a_*X,a_*X)=   - e^{2e_j(H)}
d^c f(\widetilde{[JX,X]}_z)$$    
 \begin{equation}\label{TI} \textstyle =\textstyle - e^{2e_j(H)} t \, d^c f({(\widetilde{E^j}})_{z}) = -{t\over 2} \textstyle   \frac{\partial\widetilde f}{\partial a_j}(H),\end{equation}
 for some $ t\in\R\setminus\{0\}.$  
 By Remark \ref{EXPLANATION}, one has $t>0$.  By  the comparison of (\ref{TI})  with the formula obtained in Remark\,\ref{CONSTANTS},  one deduces the  exact value of $t$, namely~$\,t = \frac{4 \Vert X\Vert^2}{{\bf b}}$ and 
 $$\textstyle  h_f(a_*X,a_*X)= h_f(a_*JX,a_*JX)=   -2 \frac{\Vert X\Vert^2}{{\bf b}} \frac{\partial\widetilde f}{\partial a_j}(H).$$
Finally, let $X=Z^\mu+ \overline {Z^\mu}$ and $ Y=Z^{\nu}+\overline {Z^{\nu}}$ be elements of the basis of  $\g^{e_j}$  given in  Remark \ref{BASIS2}\,(c), for $\mu,\, \nu\in \Delta^+$    distinct roots  satisfying  $\nu\not=\mu,\,\bar \mu$. 
Then, by   (\ref{FORMULONE}) and Lemma \ref{COMPLEXSTRUCTURE}(c)   
one has $h_f(a_*X,a_*Y)  =0$.  
This concludes the proof of (iv) and of the proposition.
\end{proof}

  \bn
${\bf Remark.}$ The usual Levi form   
 $L_f^\C$ of $f$   is given by
 $L_f^\C(Z,\overline W)=2(h_f(X,Y)+ih_f(X,JY)),$ 
where $Z=X-iJX$ and $W=Y-iJY$ are elements of type $(1,0)$. 
One easily sees that $L_f^\C$  is (strictly) positive definite if and only if $h_f$ is (strictly) positive definite.
  
\bigskip
\section{$\,N$-invariant Stein domains in $G/K$}
\label{INVARIANTSTEIN}

\medskip

The main goal of this section  is to characterize  the Stein  $\,N$-invariant domains~$D$ in $\,G/K\,$  in terms of an associated $r$-dimensional tube domain.   
We show that  $D$ is Stein if and only  if  the base of the associated tube domain  is convex and satisfies  an
additional  geometric  condition, arising   from the features of the $\,N$-invariant plurisubharmonic functions on $D$.

At the end of the  section we also prove a univalence result for $N$-equivariant Riemann domains over $G/K$. As a by-product,    a precise description of the envelope  of holomorphy of $N$-invariant domains in $G/K$ follows.

Resume the notation introduced in Section 2.  Denote by $\,R:= \exp \big( \oplus \g^{2e_j}\big)$  the  unipotent abelian subgroup  of $G$, isomorphic to $\R^r$. 
The orbit of the base point  $\,eK\in G/K\,$ under the  product of the $r$ commuting    $\,SL_2(\R)$'s contained in $G$  is   the $r$-dimensional $\,R$-invariant closed complex submanifold of $\,G/K\,$
 $$\,R \exp (\a) \cdot eK.$$ 
 By the Iwasawa decomposition of $G$, such manifold  intersects all $N$-orbits in $G/K$.
 Equivalently, 
  $$\,N \cdot (R \exp (\a)\cdot eK) =G/K.$$

The above  facts  together with the  next  proposition can be regarded as an  analogue,  for the $N$-action,  of the  polydisk theorem (cf.\,\cite{Wol72}, p.\,280).  
Denote by $\,\HH \,$ the upper half-plane in $\,\C$, with the usual $\,\R\,$-action by translations.

\mn
\begin{prop}
\label{FACT1}  
 The map  $\,{\mathcal L}:\HH^r \to R\exp\a\cdot eK$, defined  by 
 $$\textstyle \quad   (x_1+iy_1, \dots,x_r+iy_r) \to  
  \exp(\sum_j x_jE^j) \exp( \frac{1}{2}\sum_j \ln({ y_j})A_j)K\,,$$
is an equivariant biholomorphism. 
  \end{prop} 

\begin{proof} The map is clearly bijective and equivariant. To prove  that is holomorphic, it is sufficient to consider  the rank-1 case. 
Computing  separately
$$ \textstyle  d\mathcal L_zJ \ddx\big|_z=d\mathcal L_z  \ddy\big|_z = \ddt {\mathcal L}(x+i (y+t))=
\ddt \exp (xE)\exp({1\over 2}\ln (y+t)A) K $$
$$ \textstyle=\ddt \exp (xE)\exp(({1\over 2}\ln y+  \frac{t}{2y}+o(t^2))  A) K =(\exp (xE)\exp({1\over 2}\ln y A))_*\frac{1}{2y}A$$
and
$$\textstyle J\mathcal L_z  \ddx\big|_z=J\ddt {\mathcal L}(x+t+iy) =J\ddt \exp((x+t)E)\exp({1\over 2}\ln yA)  K $$
$$\textstyle =J\ddt \exp(xE)\exp(tE)\exp({1\over 2}\ln yA)  K $$
$$\textstyle =J\ddt \exp(xE)\exp({1\over 2}\ln yA) \exp(t \,Ad_{\exp (-\frac{1}{2}\ln y A)}E) K $$
$$\textstyle = J\exp(xE)_*\exp({1\over 2}\ln yA)_*\frac{1}{y}E=  (\exp (xE)\exp({1\over 2}\ln y A))_*\frac{1}{2y}A,$$

\sn
we obtain the desired identity $\textstyle  d\mathcal L_zJ \ddx\big|_z=J d\mathcal L_z \ddx\big|_z,$ for all $z\in\HH$.
\end{proof}

\sn
 \begin{remark}
 \label{FACT5}  {\rm The closed complex submanifold  $\,R \exp (\a)\cdot eK \,$ can also be regarded as the local orbit of $eK$ under the universal complexification $\,R^\C $ of $R$.
Up to a traslation, $\,\mathcal L\,$ is the local $\,R^\C$-orbit map through $\,eK$.}
\end{remark}

\medskip
As a consequence of the above biholomorphism we obtain a one-to-one correspondence between $\R^r$-invariant tube domains in $\HH^r$ and $N$-invariant domains in $G/K$.
Denote  by $\,L: \R^{>0}\times\ldots\times \R^{>0} \to \a\,$  the diffeomorphism   determined by~$\mathcal L$
\begin{equation}\label{ELLE} \textstyle L(y_1,\ldots,y_r) :=   \frac{1}{2}\sum_j \ln({ y_j})A_j. \end{equation}

 \begin{cor} \label{ASSOCIATEDTUBE}{\bf ($N$-invariant domains in $G/K$ and  tube domains in $\C^r$)}. 
 
\sn 
(i)  Let  $\,D=N\exp (\mathcal D) \cdot eK\,$ be an $\,N$-invariant domain in $G/K$  and let  $\,R\exp(\mathcal D) \cdot eK\,$ be its intersection with   the 
closed complex submanifold  $\,R\exp(\a) \cdot eK\,$.  Then the $r$-dimensional tube domain associated to $D$ is by definition  the preimage of $\,R\exp(\a) \cdot eK\,$ under $\mathcal L$, namely 
$$\,\R^r +i \Omega ,\quad \hbox{ where  $\,\Omega:= L^{-1}(\mathcal D)\,$}. $$ 

\sn
(ii) Conversely,   a   tube domain  $\,\R^r + i\Omega \,$  in $\,\HH^r \,$  determines a unique $\,N$-invariant domain  
$$\, D=N \exp(\mathcal D) \cdot eK, \quad \hbox{ where  $\mathcal D=L(\Omega)$}.$$
  \end{cor}

\sn
\begin{remark}  \label{CONVEXITY} {\rm  If  $\,D\,$ is Stein, then the associated tube domain $\R^r +i \Omega\subset \C^r$ is Stein, being biholomorphic to  the  Stein closed complex submanifold
 $\, R\exp (\mathcal D) \cdot eK\,$  of   $\,D$.  In particular, the  base $\,\Omega\,$   is an open  convex set in 
$\,(\R^{>0})^r$.

On the other hand,  already in the case of the unit ball  $\,\B^n$ in $ \C^n$, with $\,n>1$,  one can   see that the base $\Omega$ of an $N$-invariant Stein subdomain $D$  must be  an entire half-line,  and cannot be   just an arbitrary convex subset of $\R^{>0}$.    }  \end{remark}  

 The main goal of this section is to give a precise  characterization of the convex sets $\,\Omega\subset (\R^{>0})^r\,$
   arising from $N$-invariant Stein domains $D$ in~$G/K$.  As we shall see, their shape is determined by the particular  features    of the  Levi form of the $N$-invariant   functions on $D$, which involve both the Hessian and the gradient of~$\widetilde f$ (cf. Prop. \ref{LEVI}).

 Let $\,f:D \to \R\,$ be an $\,N$-invariant plurisubharmonic function. Then $f$ is uniquely determined by the function $\,\widetilde f(H):=f(\exp H\cdot eK)$ on $\,\mathcal D\,$ (cf. (\ref{EFFETILDE})) 
and also by the   function \begin{equation}\label{EFFEHAT}\,\widehat f({\bf  y}):=f(\exp (L({\bf  y}))K) =\widetilde f(L({\bf y}))\end{equation}
defined for   ${\bf y}\in \Omega,  $  as shown by    the following commutative diagram   
\begin{displaymath} 
\xymatrix@R=1.8em{ 
       \ar[d]_{L} \Omega \ar[rd]^{ \widehat f}     &       \\
 \ar[d]_{\exp} {\mathcal D}\ar[r]^{\widetilde f}      &      \R  \\
  D \ar[ru]^f&      }
\end{displaymath}
Since  the $\,N$-action on $\,D\,$ is proper and every $\,N$-orbit intersects transversally  the smooth slice  $\,\exp(L(\Omega)) \cdot eK\,$ 
in a single point, it is easy to check that the map
$\,f \to \widehat f\,$   is a bijection from the class $\,C^0(D)^N\,$ of continuous
$\,N$-invariant functions on $\,D\,$ and the class $\,C^0(\Omega)\,$ of continuous  functions on $\,\Omega$.   
By Theorem 4.1 in \cite{Fle78}, such a map  is also a bijection between $\,C^\infty(D)^N\,$ and $\,C^\infty(\Omega)$.
Analogous statements hold true for the map $\,f \to \widetilde f$.

Given a non-compact irreducible Hermitian symmetric space,  define the cone
\begin{equation} \label{CONE}\textstyle C:=\begin{cases} (\R^{>0})^r,  \hbox{ in the non-tube case,}\\ (\R^{>0})^{r-1}\times  \{0\}, \hbox{ in the  tube case.}\end{cases}\end{equation} 

\mn
The next lemma characterizes  the plurisubharmonicity of a smooth $N$-invariant function $f$ in terms of the corresponding  functions   $\widetilde f$ and $\widehat f$.

\begin{prop}
\label{PSHPOLIDISC2} 
Let $\,D\,$ be an $\,N$-invariant domain in $G/K$ and let $\,f:D \to \R\,$ be a smooth, $\,N$-invariant, plurisubharmonic function.
Then the following conditions are equivalent:

\begin{itemize}
 \smallskip
\item[(i)] $\,  f\,$ is plurisubharmonic (resp. strictly plurisubharmonic) at $z=aK$, with $a=\exp(H)$ and $H\in\mathcal D$; 

 \smallskip
\item[(ii)]
the form 
\begin{equation}
\label{LEVIA}
\textstyle  \Big (-2 \delta_{jl} \frac{\partial \widetilde f}{\partial a_l}(H) +\frac{\partial^2 \widetilde f}{\partial a_j\partial a_l}(H)\ \Big)_{j,l=1,\ldots,r}
 \end{equation}
  in Proposition \ref{LEVI}(ii)  is positive semidefinite (resp. positive definite) and 
 $$\grad \widetilde f(H) \cdot {\bf v} \le 0
  ~~\hbox{(resp. $< 0$)}, \quad \hbox{ for all ${\bf v}\in \overline C\setminus\{{\bf 0}\}$} ;$$

 \smallskip
\item[(iii)] the Hessian of $\,\widehat f\,$ is positive semidefinite (resp. positive definite)
at ${\bf y}=(y_1,\ldots,y_r)=L^{-1}(H)$ and 
\begin{equation}\label{EFFEHATDEC}\textstyle  
 \grad \widehat f({\bf y}) \cdot {\bf v} \le 0
  ~~\hbox{(resp. $< 0$)}, \quad \hbox{ for all ${\bf v}\in \overline C\setminus\{{\bf 0}\}$}. 
\end{equation}  

\end{itemize}
\end{prop}

 \medskip
 \begin{proof}
The equivalence   $(i) \Leftrightarrow (ii)$ follows directly from Proposition\,\ref{LEVI}.\pn
$(ii) \Leftrightarrow (iii)$ Since  $\,L(y_1,\ldots,y_r) =(\frac{1}{2}\ln({y_1}), \dots, \frac{1}{2}\ln({ y_r} ))\,$ (see (\ref{ELLE})),
one has 
  $\widetilde f (a_1,\ldots, a_r) = \widehat f(e^{2a_1} ,\ldots,e^{2a_r} ) \,.$
Therefore
\begin{equation}
\label{DEREFFETILDE}
\textstyle 
\frac{\partial \widetilde f}{\partial a_j }(a_1,\dots, a_r)=
2\frac{\partial \widehat f}
{  \partial y_j } {\scriptstyle (e^{2a_1},\ldots,e^{2a_r} )}e^{2 a_j}
\end{equation}
\begin{equation} \label{HESSIANEFFETILDE}\textstyle 
\frac{\partial^2 \widetilde f}{\partial a_j   \partial a_l}(H)=   
4\frac{\partial^2 \widehat f}
{  \partial  y_j   \partial  y_l}{\scriptstyle (e^{2a_1},\ldots,e^{2a_r} )}e^{2 a_j}e^{2 a_l}
+ 4
\frac{\partial \widehat f}
{  \partial  y_j }{\scriptstyle (e^{2a_1} ,\ldots,e^{2a_r} )} e^{2 a_j}\delta_{jl}\,.\end{equation}
By combining formulas (\ref{DEREFFETILDE})  and
(\ref{HESSIANEFFETILDE})   one obtains
\begin{equation}
\label{HESSEFFETILDE}
\textstyle\big (4 \frac{\partial^2  \widehat f}{\partial  y_j   \partial
y_l}e^{2 a_j}e^{2 a_l}
\big )_{j,l}=\big ( \frac{\partial^2 \widetilde f}{\partial   a_j  
\partial a_l}-2\,\frac{\partial  \widetilde f}{\partial  
a_j }\delta_{jl}  \big )_{j,l}.
\end{equation}
Also, by (\ref{DEREFFETILDE}),  the same monotonicity conditions hold both for  $\widetilde f$ and  for $\widehat f$.
  \end{proof}

\mn
\begin{defi}\label{STABLYCVX}
A smooth function $\,g\colon \R^r\to\R\,$ is  convex (resp. stably convex) if
its Hessian  is   semidefinite (positive definite).
\end{defi}

\sn
\begin{remark}
 \label{FACT23} {\rm The above lemma shows that the function $\widehat f$ corresponding to a smooth  $N$-invariant plurisubharmonic function is not just an arbitrary smooth convex function, but it must satisfy the additional monotonicity conditions (\ref{EFFEHATDEC}).
 (cf. Rem.\,\ref{DECGRADIENT}).}
 {\rm
\REM{ It is well known that an 
 $\,\R^r$-invariant smooth function  on a tube domain  $\,\R^r +i \Omega\,$ is (strictly) plurisubharmonic if and only its restriction to the base $\Omega$ is     
 (stably) convex. Hence the above equivalent conditions can be reformulated as follows:
  the restriction of the smooth, $N$-invariant function $f$ to the embedded tube domain $\,R\exp (\mathcal D) \cdot eK \cong \R^r +i \Omega\,$ is (strictly) plurisubharmonic. 
As we shall see in Section 5, such a restriction must necessarily satisfy some additional conditions.}}
\end{remark}

\sn
\begin{defi} \label{CONEINVARIANT} 
A set $\,\Omega \subset \R^r\,$ is  
$C$-invariant if  $\,{\bf y}\in \Omega\, $ implies $\,{\bf y}+C \subset  \Omega $
Equivalently, if  $\,{\bf y}\in \Omega\, $ implies $\,{\bf y}+\overline C \subset  \Omega,$ where $\overline C$ denotes the closure of $C$.
\end{defi}

\sn\begin{theorem}
\label{CASOLISCIO1} 
Let $\,G/K\,$ be   a non-compact irreducible Hermitian symmetric space  and let $D$ 
 be an $\,N$-invariant domain 
in $\,G/K$. 
Then $\,D\,$ is Stein if and only if  the base $\,\Omega\,$ of the associated tube domain is convex and $C$-invariant.
\end{theorem}
 
The proof of  the above theorem is divided into two parts. If  $D$ has smooth boundary,   then  the argument relies
 on  the computation of the Levi form of smooth, $\,N$-invariant functions  on $\,D\,$ (Prop.\,\ref{LEVI}) and  some 
elementary convex-geometric  properties of $\,\Omega$.

 In the general case, the proof of the theorem is obtained by realizing  $\,D\,$ as an increasing union of  Stein, $\,N$-invariant domains with smooth boundary. 

 \nmedskip
 {\bf Proof of Theorem \ref{CASOLISCIO1}: the smooth case}. 
 The rank-1 tube case is trivial, since  every  $\R$-invariant domain in the upper half-plane $\HH$ is Stein. 
 So we deal with the remaining cases:
 the rank-one non-tube case   and the higher rank cases. 
 
We  use the notation ${\bf y}=(y_1,\ldots,y_r)$, for elements in $\R^r$.  
Let $\,D\subset G/K\,$ be a Stein, $\,N$-invariant   domain with smooth boundary  and let $\R^r+i\Omega\subset \C^r$ be its associated tube domain.
Then $\,\Omega\,$ is a convex set with  smooth boundary 
(cf.\,Rem.\,\ref{CONVEXITY}). 
 Assume by contradiction that $\Omega$ is not $C$-invariant, i.e.\,there exist $\,{\bf y} \in \Omega\,$ and $ {{\bf z}}\in 
 ({\bf y}+C)\cap\partial\Omega$. By the convexity of $\Omega$,  the open segment from $\bf y$ to $ {\bf z} $ is contained in $\Omega$. In addition,   the vector
 ${\bf v} ={{\bf z}-\bf y} \in C$  
  is transversal to the  tangent hyperplane 
 $T_{ {\bf z}}\partial \Omega$ 
 and points outwards. 
 Therefore, given a smooth local defining function 
 $\widehat f$ of $\partial \Omega$ near $ {\bf z}$, one has 
 $$\textstyle \frac{\partial \widehat f}{\partial {\bf v}}({\bf  z})=\grad \widehat f({\bf  z})\cdot {\bf v}>0.$$
In the tube case, the above inequality and   (\ref{DEREFFETILDE}) imply that $\frac{\partial \widetilde f}{\partial a_j}({H})>0$,   for some $j \in \{1,\dots, r-1\}$. Then, by  Proposition \ref{LEVI}\,(iii), the Levi form of the corresponding $N$-invariant function $f$ is negative definite on  the  $J$-invariant subspace $a_* \g^{e_j-e_l}\oplus a_* \g^{e_j+e_l}$ of $T_{aK}(\partial D)$, the tangent space to $\partial D$ in $aK$.
 In the non-tube case, one has $\frac{\partial \widetilde f}{\partial a_j}({H})>0$,   for some $j \in \{1,\dots, r\}$.  By   Proposition \ref{LEVI}\,(iv), the Levi form of the corresponding $N$-invariant function $f$ is negative definite on  the  $J$-invariant subspace  $ a_* \g^{e_j}$ of  $T_{aK}(\partial D)$.
This contradicts the fact that $f$ is a defining function of the Stein $N$-invariant domain $D$ and proves that $\Omega$ is $C$-invariant.

Conversely, assume that $\Omega$ is convex and $C$-invariant.  We  prove that  $D$ is Stein by showing that it is Levi-pseudoconvex, i.e.   for all points $aK\in \partial D$ and  local defining functions $f$ of $ D$ near $aK$,  one has $h_f(X ,X) \ge 0$,  for every  tangent vector~$X\in T_{aK}\partial D\cap J T_{aK}\partial D$, the complex tangent space to $\partial D$ at~$aK$. 

Let ${\bf z}\in \partial \Omega$ and let $aK= \mathcal L ({\bf z})$. Denote by  $W:= T_{{\bf z}}\partial \Omega$ the tangent space to $\partial \Omega$ in ${\bf z}$. 
One can verify that the complex tangent space to $\partial D$ at~$aK$ is given by 
 $$ a_* (\bigoplus \g^{e_j\pm e_l}\oplus \bigoplus \g^{e_j})\oplus (\mathcal L_*)_{\bf z}W\oplus J(\mathcal L_*)_{\bf z}W.$$

Let ${\bf v}=(v_1,\ldots,v_r)$ be an outer normal vector to $W$ in $\R^r$.  The  $C$-invariance and the convexity of $\Omega$ imply  that $v_j\le 0$, for $j=1,\ldots,r$ in the non-tube case, and  $v_j\le 0$, for $j=1,\ldots,r-1$ in the  tube case.  Otherwise  the space $W$
would intersect $\,{\bf y} +C\,$, for every $\,{\bf y}\in \Omega$, yielding a contradiction. 

 Let $\widehat f$ be a smooth local defining function of $\Omega$ near ${\bf z}$. By the  convexity of $\Omega$,  the Hessian  $\,Hess(\widehat f)({\bf z})\,$ is positive definite on $W$. Moreover,  as the gradient  $\grad \widehat f({\bf z})$ is a positive multiple of $\bf v$, one has  $\frac{\partial \widehat f}{\partial y_j}({\bf z})\le 0$, for  all $j=1,\ldots,r$, in the non-tube case, and  
 $\frac{\partial \widehat f}{\partial y_j}({\bf z})\le 0$, for  all $j=1,\ldots,r-1$, in the tube case.
 
Let $f$ be the corresponding $N$-invariant  local defining function of  $ D$  near  $aK= \exp L( {\bf z})K$. 
 By Proposition \ref{PSHPOLIDISC2},  the Levi form of $f$ is positive definite on  $(\mathcal L_*)_{\bf z}W\oplus J(\mathcal L_*)_{\bf z}W\subset a_*\a\oplus a_*J\a$.

In addition, by (\ref{DEREFFETILDE})  
and Proposition \ref{LEVI},  the Levi form of $f$ is positive definite on  $a_* (\bigoplus \g^{e_j\pm e_l}\oplus \bigoplus \g^{e_j})$.
As a result,  $D$ is Levi pseudoconvex in $ aK=\exp L( {\bf z})K$.  Since $aK$ is an arbitrary point in $\partial D\cap \exp\a\cdot eK$  and both $D$ and $f$ are $N$-invariant, the domain $D$ is Levi-pseudoconvex and therefore Stein, as desired.

\bigskip
 In order to prove Theorem \ref{CASOLISCIO1} in the non-smooth case,  
 we need some preliminary Lemmas.


\sn
\begin{lem}
\label{INCLUDED} 
Let $D$ be a  domain in a Stein manifold, let $D'\subset D$ be a subdomain  with smooth boundary
and let $z \in \partial D \cap \partial D'$. If $D'$ is not Levi pseudoconvex in $z$, then $D$ is not Stein. 
\end{lem}

 \begin{proof}
Under our assumption, there exists a one dimensional complex submanifold $M$ through $z$ in $\,X\,$ with 
$M\setminus \{z\} \subset D'$ (\cite{Ran86}, proof of Thm.\,2.11, p.\,56). This implies that $D$ is not Hartogs pseudoconvex
(\cite{Ran86}, Thm.\,2.9, p.\,54) and  in particular  it is   not Stein.
 \end{proof}


\sn

For a domain $\Omega $  in $\R^r$, denote by $d_\Omega\colon \Omega\to\R$ the distance function from the boundary (if ${\bf z}\in\Omega$, then  $d_\Omega({\bf z})$  is by definition the  radius of the largest ball centered in ${\bf z}$  and contained in $\Omega$).  The next lemma is a  known characterization of convex domains.
 
\begin{lem}
\label{LOG} 
A proper subdomain $\Omega$ of $\,\R^r\,$ is convex if and only if the function $- \ln d_\Omega: \Omega \to \R$ is convex.
 \end{lem}

\bn
In  what follows, 
for a fixed domain $\Omega$ in $\R^r$, we  denote $$u:= -\ln d_\Omega.$$  

Denote  by $\,\B_\rho({\bf y})\,$    the open ball of center ${\bf y}=(y_1,\ldots,y_r)\in\R^r$ and radius~$\,\rho$. 
Fix a smooth, positive,  radial function $\sigma: {\R^r} \to \R$ (only depending  on  $R^2=\Vert{\bf  w}\Vert^2$),
with support in
$\B_1({\bf 0})$,  such that $\sigma'(R^2)<0$ and  $\int_{\R^r} \sigma ({\bf  w})d{\bf  w} =1$.
For $\varepsilon >0$,  define $\Omega_\varepsilon:=\{{\bf  y} \in\Omega \ : \ d_\Omega({\bf  y})>\varepsilon \}$
and   $u_\varepsilon:\Omega_\varepsilon \to \R$ by 
$$\textstyle u_\varepsilon ({\bf  y}):={1\over {\epsilon^r}}\int_{\R^r} u({\bf z}) \sigma({{{\bf z-y}}\over \epsilon})d{\bf z}  =\int_{\R^r} u({\bf  y} +\varepsilon {\bf  w}) \sigma({\bf  w})d{\bf  w}       \,.$$
 The functions $u_\varepsilon$ are clearly smooth. Let $\,\nu: (\R^{>0})^r
 \to \R^{>0}$ be  the stably convex positive function  given by 
$\nu({\bf  y}) := \sum_j \frac{1}{y_j}$.
Define $v_\varepsilon:\Omega_\varepsilon \to \R$ by
$$\textstyle v_\varepsilon ({\bf  y}):=u_\varepsilon ({\bf  y})+\varepsilon \nu({\bf  y})\,.$$


\smallskip
\begin{lem}
\label{APPROSSIMANTI} 
Let $\Omega $ be a convex, $C$-invariant  domain in $\,(\R^{>0})^r$. Then the following facts hold true:
\begin{itemize}
\smallskip
\item[(i)] The domain $\Omega_\varepsilon$ is convex and $C$-invariant for every $\,\varepsilon>0$.
\smallskip
\item[(ii)] The smooth functions $v_\varepsilon$ are stably convex and, for $\varepsilon \searrow 0$, 
they decrease to $u$  uniformly on the  compact subsets of $\Omega$.
\smallskip
\item[(iii)] Let $\delta_\varepsilon:= -\ln 3\varepsilon$. The sublevel set  $\widetilde \Omega_\varepsilon := \{{\bf  y} \in 
 \Omega_\varepsilon \ : \ v_\varepsilon({\bf  y}) < \delta_\varepsilon \}$
is convex and  $C$-invariant.
\smallskip
\item[(iv)] The boundary of $\,\widetilde \Omega_\varepsilon\,$ in $(\R^{>0})^r$ coincides with $\,\{{\bf  y} \in \Omega_\varepsilon\ 
: \ v_\varepsilon({\bf  y}) = \delta_\varepsilon\,\}\,$  and it is smooth.
\smallskip
\item[(v)] As $\,n\in\N\,$ increases, the sequence of convex, $C$-invariant  subdomains with smooth boundary  $\,\widetilde \Omega_{1/n}\,$    exhausts $\, \Omega.$
\end{itemize}

\end{lem}

 \medskip
 \begin{proof}
 (i) Let $\,{\bf  y}\,$ and $\,{\bf  y} +{\bf  v}\,$ be elements  of $\,\Omega_\varepsilon$. Then $\,\B_ \varepsilon({\bf  y})\,$ and $\,\B_ \varepsilon({\bf  y+ v})\,$ 
 are contained in $\,\Omega\,$ and, by the convexity of $\Omega$,  the same is true for  $\,\B_ \varepsilon({\bf  y+ tv})$,    for every
 $\,t\in [0,1]$. This shows that $\Omega_\varepsilon$ is convex. Moreover, as $\Omega$ is $C$-invariant, if 
 $\,\B_ \varepsilon({\bf  y})\,$ is contained in $\Omega$ and ${\bf  v}$ is an element of  the cone $C$, then also the open ball
 $\,\B_ \varepsilon({\bf  y+ v})\,$ 
 is  contained in $\Omega$. This shows that $\Omega_\varepsilon$ is $C$-invariant.
 
  (ii) As $u$ is convex, for ${\bf y}$, ${\bf y}+{\bf v}\in \Omega$ and 
   $\,t\in [0,1]$, one has 
$$u_\varepsilon ({\bf  y}+t{\bf  v}):=\int_{\R^r} u({\bf  y}+t{\bf  v} +\varepsilon {\bf  w}) \sigma({\bf  w})d{\bf  w}$$
$$ \leq
\int_{\R^r} \big( (1-t)u({\bf  y} +\varepsilon {\bf  w})+tu({\bf  y} +\varepsilon {\bf  w}+{\bf  v})\big ) \sigma({\bf  w})d{\bf  w} = 
 (1-t)u_\varepsilon ({\bf  y}) + tu_\varepsilon ({\bf  y}+{\bf  v})\,,$$
showing that the smooth function $u_\varepsilon$ is convex. 
Since $\nu$ is smooth and stably convex, it follows that 
$v_\varepsilon:= u_\varepsilon+\varepsilon \nu$ is smooth and stably convex. 
Moreover, as  convexity  implies  subharmonicity,  then   the last part
of  statement (ii) follows from \cite{Hor94}, Thm 3.2.3(ii), p.143.

  (iii) Since the function $v_\varepsilon$ is convex,  then the domain $\widetilde \Omega_\varepsilon$ is convex.
 In order to show that  $\widetilde \Omega_\varepsilon$ is $C$-invariant, we prove that 
 \begin{equation}
\label{STRICTDEC}
v_\varepsilon({\bf  y+v})< v_\varepsilon({\bf  y})\,,
\end{equation}
 for every ${\bf  y}\in \Omega_\varepsilon$ and ${\bf v}\in C$.
Since $\Omega$ is $C$-invariant,   if for some ${\bf y}\in \Omega$ the ball $\B_r({\bf y})$ is contained in $\Omega$,   then also   the ball  $\B_r({\bf y}+{\bf v})$ is contained in $ \Omega$, for all~${\bf v}\in C$.  It follows that $d_\Omega({\bf y})\le  d_\Omega({\bf y}+{\bf v})$ and consequently  $u({\bf  y}+{\bf  v} +\varepsilon {\bf  w}) \le  u({\bf  y} +\varepsilon {\bf  w}) $,  for all~${\bf v}\in C$. and~${\bf  w} \in \B_1({\bf  0})$.
One deduces that  
   $$u_\varepsilon ({\bf  y}+{\bf  v})=\int_{\R^r} u({\bf  y}+{\bf  v} +\varepsilon {\bf  w}) \sigma({\bf  w})d{\bf  w} \leq
   \int_{\R^r} u({\bf  y} +\varepsilon {\bf  w}) \sigma({\bf  w})d{\bf  w}=
u_\varepsilon ({\bf  y})\,,$$
for every ${\bf  y}\in \Omega_\varepsilon$, ${\bf  v}\in C$.  
Since    $\nu({\bf  y+v})< \nu({\bf  y})$, one concludes   that $v_\varepsilon({\bf  y+v})< v_\varepsilon({\bf  y})$, and 
 $\widetilde \Omega_\varepsilon$ is $C$-invariant, as desired.
      
    (iv) For $\,{\bf y}\,$ close to $\partial \Omega_\varepsilon=
    \,\{{\bf  z} \in \Omega
: \ d_\Omega({\bf  z}) = \varepsilon\,\}, $ a   rough extimate shows that   
$d_\Omega({\bf  y }+\varepsilon {\bf  w})<3\varepsilon$, for every ${\bf  w} \in \B_1({\bf  0})$.
Therefore $v_\varepsilon ({\bf  y})>u_\varepsilon ({\bf  y})> -\ln 3\varepsilon$, implying that the boundary of $\, \widetilde \Omega_\varepsilon$ is contained in $\Omega_\varepsilon $ and it is given by 
 $\,\partial\widetilde \Omega_\varepsilon=\{{\bf  y} \in \Omega_\varepsilon \ 
: \ v_\varepsilon({\bf  y}) = \delta_\varepsilon\,\}$.
Concerning the smoothness of $\,\partial\widetilde \Omega_\varepsilon$,  the  rank one case is trivial.  So assume $r>1$.

 Let $ {\widehat {\bf y}}  \in \partial \widetilde \Omega_\varepsilon$. 
Set ${\bf  v}:=(1,\ldots ,1)$, in the non-tube case, and  ${\bf  v}:=(1,\ldots ,1,0)$, in the tube case.  Since ${\bf v}$ lies in the cone $C$, the inequality (\ref{STRICTDEC}) implies that for $\gamma$
small enough the real 
 function $g:(-\gamma, \gamma) \to \R$,  defined by $g(t):=v_\varepsilon( {\widehat{\bf y}}+t{\bf  v})$,
 is strictly decreasing. By the stable convexity of $v_\varepsilon$, it is also stricltly convex and
 $g'(0)<0$.   
As $g'(0)$ is a directional  derivative of  $v_\varepsilon$ in $ {\widehat {\bf y}}$,
 the differential  $dv_\varepsilon|_{\widehat {\bf y}}$ does not vanish and  the boundary of 
$\widetilde \Omega_{\varepsilon}$ is smooth.

    (v) For  $\,m>n\,$,   the  inclusion $\Omega_{1/n} \subset  \Omega_{1/m}\,$  and the inequality 
    $\, v_{1/n} > v_{1/m}$  imply  that $\widetilde \Omega_{1/n} \subset  \widetilde \Omega_{1/m}$. 
    This concludes  the proof of the lemma.
 \end{proof}

 \nmedskip
 {\bf Proof of Theorem \ref{CASOLISCIO1}: the general case}.
 Let $D$ be an arbitrary  Stein, $\,N$-invariant domain in $G/K$. By Remark \ref{CONVEXITY}, the base $\Omega$ of the associated tube domain   is necessarily convex.   Assume by contradiction that $\Omega$ is not $C$-invariant (cf. Def.\,\ref{CONEINVARIANT} and (\ref{CONE})),  i.e.\,there exist $\,{\bf y}\in \Omega\,$
 and $\,{\bf z} \in  ({\bf y}+C)\cap  \partial \Omega$. By the convexity of $\Omega$,  the open segment from $\bf y$ to ${\bf  z} $ is contained in $\Omega$.  Moreover,   the vector  ${\bf v} = {\bf  z}-{\bf y}  $  lies in the cone $C$ and  points to the exterior of $\Omega$.  
Let $\B_\varepsilon({\bf  y})$ be a  relatively compact ball   in $\Omega$ and define
$$t_{\max} := \max \{ \, t >0 \ : \ \B_\varepsilon({\bf  y} + t {\bf v}) \subset \Omega \, \}\,.$$
Then there exists $ {\bf  w} \in \partial \B_\varepsilon({\bf y} + t_{\max} {\bf v})\cap \partial \Omega$, and 
by construction  
$$\langle{\bf  w}-({\bf  y} + t_{max} {\bf v}),{\bf  v}\rangle\  >\ 0.$$ 
This  implies that   the outer normal $\,{\bf n}:={\bf  w}-({\bf  y} + t {\bf v})\,$ to $\,\partial \B_\varepsilon({\bf  y} + t_{\max} {\bf v})\,$ satisfies $n_j>0$, for some $j\in\{1,\ldots,r\}$ in the non-tube case (resp. $n_j>0$,  for some $j\in\{1,\ldots,r-1\}$, in the  tube case). 
From the result of the theorem in the smooth case,   it follows that the $\,N$-invariant subdomain  $\,N  \exp(L(\B_\varepsilon({\bf  y} + t_{\max} {\bf v})))\cdot eK$, with smooth boundary, is not  Levi pseudoconvex 
in $ \,\exp (L(  {\bf  w}))K$. Then Lemma \ref{INCLUDED} implies that $\,D\,$ is not Stein, contradicting the assumption.

Conversely, assume that $\,\Omega\,$ is  convex and $C$-invariant.
By Lemma \ref{APPROSSIMANTI},    the domain $D$ can be realised as the increasing union 
 of $N$-invariant domains $\,D_{1/n}:=N\exp(L (\widetilde \Omega_{1/n}))\cdot eK$, where the open  sets 
 $\,\widetilde \Omega_{1/n}\subset  \R^r  \,$ are convex, $C$-invariant and have smooth boundary.
By  the result of the theorem  in the smooth case,  the domains $D_{1/n}$ are Stein and so is
their increasing union $\,D$. This completes  the proof of the theorem.
\qed

 \bigskip

We conclude this section with a univalence result for  Stein, $\,N$-equivariant, Riemann domains over $G/K$.

 \smallskip
\begin{prop}
\label{SCHLICHT}
Any  holomorphically separable, $\,N$-equivariant, Riemann domain over $\,G/K\,$ is univalent. 
\end{prop}

\begin{proof} Let $\,Z\,$ be a  holomorphically separable,  $\,N$-equivariant, Riemann domain  over $\,G/K\,$. 
 By \cite{Ros63},  $\,Z\,$ admits an   holomorphic,   $\,N$-equivariant  open embedding  into its envelope of holomorphy, which is a Stein $\,N$-equivariant, Riemann domain over $\,G/K\,$.
  Hence, without loss of generality, we may assume that
 $\,Z\,$ is Stein.

Denote by $\,\pi: Z \to G/K\,$ the $\,N$-equivariant projection
and let $\,\pi(Z)= N \exp (L(\Omega))\cdot eK$ be the image of $Z$ under $\pi$. Define 
$\,\Sigma :=\exp (L(\Omega))\cdot eK\,$ and
$\widetilde \Sigma:=\pi^{-1} (\Sigma)$. Note that  $\,\widetilde \Sigma\,$ is a closed submanifold of
$\,Z$.  

\nmedskip
{\bf Claim.} {\sl The map
$\,\widetilde {\phi}:N \times \widetilde \Sigma \to Z$, given by $(n,x)  \to n\cdot x$, is a diffeomorphism. }

\nmedskip
{\it Proof of the  claim.} Since $\,\Sigma= \pi(Z)\cap \exp(\a)\cdot eK\,$ is a closed real submanifold of $\,\pi(Z)\,$ and $\,\pi\,$ is a local biholomorphism, 
the restriction $\,\pi|_{\widetilde \Sigma}: \widetilde \Sigma \to
\Sigma\,$ is a local diffeomorphism. Moreover one has the commutative diagram 
\begin{displaymath}
\xymatrix@R=1.8em{
       \ar[d]_{^{Id \times (\pi|_{\widetilde \Sigma})} } N\times \widetilde\Sigma \ar[r]^{ \widetilde\phi }     &    Z \ar[d]^\pi  \\
 N\times  \Sigma\ar[r]^{\phi}  &    N\exp L(\Omega) \cdot eK  \\
  }
\end{displaymath}
where the maps $\,Id \times (\pi|_{\widetilde \Sigma})$, ${\phi}$ and $\pi$ are local diffeomorphisms. Hence so is the map 
$\widetilde {{\phi}}$.

To prove that $\widetilde {{\phi}}$ is surjective, 
let $z\in Z$ and note that  $\pi(z) 
=n\exp(L({\bf y}))K$, for some $n\in N$ and ${\bf y} \in \Omega$. Then the element $w:=n^{-1}\cdot z \in \widetilde  \Sigma$ satisfies  $n\cdot w=z$, implying  the surjectivity of $\widetilde {{\phi}}$.

To prove that $\widetilde {{\phi}}$ is injective, assume  that  $\,n\cdot w=n'\cdot w'$, for some  $n,\, n'\in N\,$ and $\,w,\,w' \in \widetilde \Sigma$. From the equivariance of $\,\pi\,$  it follows that $\,n\cdot \pi(w)=n'\cdot \pi(w')$.
As $\,\phi\,$ is bijective,
it follows that $n=n'$  and $\pi(w)=\pi(w')$. Thus $w=(n^{-1}n') \cdot w'=w'$, implying the injectivity of $\widetilde {{\phi}}$ and concluding the proof of the claim.

 \medskip
Now, in order to prove the univalence of $\,\pi$, it is sufficient to show that the restriction
$\,\pi|_{ \widetilde \Sigma}: \widetilde  \Sigma \to \Sigma\,$ of $\,\pi\,$ to $\, \widetilde  \Sigma\,$ is injective.
For this, consider the closed complex submanifold $\,R \cdot \widetilde \Sigma=
\pi^{-1}(R\cdot \Sigma)$ of $\,Z$. As $\,Z\,$ is Stein, so is $\,R \cdot \widetilde \Sigma$.
Hence the restriction
 $\pi|_{R\cdot \widetilde \Sigma}:R\cdot \widetilde \Sigma \to R\cdot \Sigma\,$ 
 defines  an $\,R$-equivariant,  Stein, Riemann domain over 
the Stein tube   $\,R\cdot \Sigma$. As $\,R\,$ is isomorphic to $\,\R^r$,  from  \cite{CoLo86}
it follows that $\,\pi|_{R\cdot \widetilde \Sigma}\,$ is injective. Hence the  same is  true for  $\,\pi|_ {\widetilde \Sigma}\,$
and $\,\pi$, as wished.
 \end{proof}

\begin{cor} \label{ENVELOPE}
The envelope of holomorphy $\widehat D$ of an
$\,N$-invariant domain $\,D\,$ in $\,G/K\,$ is the smallest Stein domain in $G/K$ containing $D$.  More precisely, $\widehat D$ is the tube domain with base $\widehat \Omega$,  the convex $C$-invariant hull of $\,\Omega$.
\end{cor}


\bigskip
\section{$\,N$-invariant psh functions vs. cvxdec functions}
\label{PSHEXPCVX}

\medskip

Let $\,D\,$ be a Stein, $\,N$-invariant domain in a
 non-compact, irreducible  Hermitian symmetric space  $\,G/K\,$ of rank $\,r$ and 
let $\,\Omega\,$ be the base of the associated $\,r$-dimensional tube domain.
Then $\,\Omega\,$ is a convex, $\,C$-invariant domain in $\,(\R^{>0})^r\,$ (Thm.\,\ref{CASOLISCIO1}).
From Proposition \ref{PSHPOLIDISC2}  it follows that 
there is a one-to-one correspondence between the class  of smooth 
$\,N$-invariant plurisubharmonic functions on $\,D\,$ and the class  of smooth convex functions on $\,\Omega\,$
satisfying an additional monotonicity condition  (cf. Rem.\,\ref{FACT23} and Rem.\,\ref{DECGRADIENT}).
In this section we obtain an analogous result in the non-smooth 
context.

\mn
Let $\,\overline C\,$ be the closure of the cone defined in (\ref{CONE}).

  \medskip
  \begin{defi} \label{DECREASING} 
A function $\,\widehat f:\Omega \to \R\,$ is (strictly) $\,\overline C$-decreasing if for every $\,{\bf y} \in \Omega\,$ and
$\,{\bf v} \in \overline C\setminus\{{\bf 0} \} \,$ the restriction of $\,\widehat f\,$ to the half-line $\,\{ {\bf y} + t {\bf v } \ : \ t \geq 0 \}\,$ is (strictly) decreasing.
\end{defi}

\mn
\begin{remark}
 \label{DECGRADIENT}
 {\rm (i) A smooth function $\,\widehat f:\Omega \to \R\,$ is $\,\overline C$-decreasing if and only if $\,{\rm grad}f({\bf y}) \cdot {\bf v} \leq 0\,$
 for every $\,{\bf y} \in \Omega\,$ and $\,{\bf v} \in \overline C\setminus\{{\bf 0} \}$.}
 
 \sn
  (ii) {\rm A smooth, {\it stably convex} (cf. Def. \ref{STABLYCVX}) function $\,\widehat f:\Omega \to \R\,$ is $\,\overline C$-decreasing if and only if $\,{\rm grad}f({\bf y}) \cdot {\bf v} < 0 $, 
 for every $\,{\bf y} \in \Omega\,$ and $\,{\bf v} \in \overline C\setminus\{{\bf 0} \}$. This follows from the fact that the directional  derivatives $\,{\rm grad}f({\bf y}) \cdot {\bf v}\,$  of a stably convex, $\,\overline C$-decreasing function $\,\widehat f$  never vanish. In particular
 $\,\widehat f\,$ is automatically strictly $\,\overline C$-decreasing.}
\end{remark}

\medskip
\noindent
In view of the above observations, we define the following classes of functions:

 \medskip
- $ConvDec^{\infty,+}(\Omega)$: smooth, stably convex, $\,\overline C$-decreasing
functions  on $\,\Omega$,

 \smallskip
- $ConvDec^{\infty}(\Omega)$: smooth, convex,
$\, \overline C$-decreasing
functions  on $\,\Omega$,

\smallskip
- $Psh^{{\infty},+}(D)^N$: smooth,  $\,N$-invariant,  strictly plurisubharmonic functions
on~$\,D$,

\smallskip
- $Psh^{{\infty}}(D)^N$: smooth,  $\,N$-invariant,  plurisubharmonic functions
 on~$\,D$.

\mn
Proposition \ref{PSHPOLIDISC2} established a one-to-one correspondence between $ConvDec^{\infty,+}(\Omega)$
and   $Psh^{{\infty},+}(D)^N$, as well as between $ConvDec^{\infty}(\Omega)$ and $Psh^{{\infty}}(D)^N$.
The next  goal is to extend  such correspondences beyond the smooth  context.  

Let $\,\widehat h:\Omega \to \R\,$ be the smooth, stably convex,  strictly $\overline C$-decreasing function  
\begin{equation}\label{H-HAT}
\textstyle \widehat h({\bf y}):=\sum_j \frac{1}{y_j}, \qquad \hbox{for ${\bf y}=(y_1,\ldots,y_r)\in \Omega$}, 
\end{equation}
and let   $\,h\,$ be  the $\,N$-invariant strictly plurisubharmonic function on $D$ associated to  $\,\widehat h$.

\begin{defi} A  function $\,\widehat f\,\colon \Omega\to\R$ is    {\it stably convex and  $\,\overline C$-decreasing }  
if 
every point in $\,\Omega\,$ admits 
a convex $\,\overline C$-invariant neighborhood $W$ and $\,\varepsilon >0$ such that $\widehat f-\varepsilon \widehat h$ is a 
 convex, $\, \overline C$-decreasing function on $W$.  
\end{defi}

\smallskip
\begin{defi}  An $\,N$-invariant function $\,f\colon D\to \R\,$ is    {\it  strictly plurisubharmonic} if 
 every point  in  $\, D$ admits 
an $\,N$-invariant neighborhood $\,U$ and $\,\varepsilon >0$\, such that $\,f-\varepsilon h\,$ is an $N$-invariant plurisubharmonic function on~$U$ (see also \,\cite{Gun90}, Vol.\,1, Def.\,1, p.\,118).  \end{defi}
\sn
 
In the smooth context the above  notions coincide  with the ones introduced earlier.
Denote by 

\smallskip
- $ConvDec^+(\Omega)$:  
stably convex and $\,\overline C$-decreasing functions on $\Omega$;

\smallskip
-  $ConvDec(\Omega)$:  convex,  $\,\overline C$-decreasing functions on $\Omega$;

 \smallskip
- $Psh^+(D)^N$: strictly plurisubharmonic, $\,N$-invariant functions on $\,D$;

\smallskip
- $Psh(D)^N$: plurisubharmonic, $\,N$-invariant functions on $\,D$.

\REM{
\medskip
\begin{lem}
\label{MOLLIFICA}
 If $\,\widehat f\,\in \,ConvDec(\Omega)$,  then for all $\varepsilon>0$, the functions  $\widehat f_\varepsilon $ are in  $\,ConvDec^{\infty,+}(\Omega)$.
 As a conseguence,   the corresponding  functions $\,f_\varepsilon\,$ on $D$ are in  $ \,Psh^{\infty,+}(D)^N\,$ and the limit function 
  $\,f\,$ is in  $\,Psh(D)^N$.
\end{lem}
\smallskip
\begin{proof}
Arguments analogous  to those used  in Lemma \ref{APPROSSIMANTI} show that the functions 
$\,\widehat f_\varepsilon\,$ are    in the class $\,ConvDec^{\infty,+}(\Omega_\varepsilon)$. 
Consequently, $\grad \widehat f_\varepsilon \cdot C <0$.
By Proposition \ref{PSHPOLIDISC2}, 
the corresponding functions $\,f_\varepsilon\,$ are in  $\,Psh^{\infty,+}(D)^N\,$ and consequently  
  $\,f\,$ is in  $\,Psh(D)^N$, as claimed.
 \end{proof}}

 \bn
The next theorem summarizes our results.
 
\medskip
\begin{theorem}
\label{BIJECTIVEK} Let $D$ be a Stein $\,N$-invariant domain in a non-compact, irreducible 
Hermitian symmetric space $G/K$ of rank $r$.
The map $f\to \widehat f$ is a bijection between
the following classes of functions 
\begin{itemize}
\item [(i)] $\,Psh^{{\infty},+}(D)^N\,$ and $\ \ \,ConvDec^{{\infty},+}(\Omega)$,
\item [(ii)] $\,Psh^{\infty}(D)^N\,\ \ $ and $\ \ \,ConvDec^{\infty}(\Omega)$,
\item [(iii)] $\,Psh\,(D)^N\,\ \ \ \,$ and $\, \ \  ConvDec(\Omega)$,
\item [(iv)]  $\,Psh^{+}(D)^N\,\ \ $ and $\ \ \,ConvDec^+(\Omega)$.
\end{itemize}

\noindent
In particular, $\,N$-invariant plurisubharmonic functions on $D$ are necessarily continuous.
\end{theorem}

\begin{proof} (i) and (ii) follow from Proposition\,\ref{PSHPOLIDISC2} and Remark \ref{DECGRADIENT}.\pn
(iii) Let $\,f \,$ be a function in  $\,Psh\,(D)^N$.  
Since the restriction of  $\,f\,$ to the embedded $r$-dimensional Stein tube domain 
 $\,R \exp(L(\Omega)) \cdot eK \cong \R^r \times i\Omega\,$ (cf.\,Cor.\,\ref{ASSOCIATEDTUBE}) is plurisubharmonic and $\,R$-invariant,
then $\,\widehat f$ is necessarily  convex. Assume by contradiction that $\,\widehat f\,$ is not $\,\overline C$-decreasing. 
  Then there exists  $\,s\in \R\,$ such that the sublevel set 
$\,\{\widehat f<s\}\,$ is not $\,\overline C$-invariant. By Theorem \ref{CASOLISCIO1}, the corresponding $\,N$-invariant domain 
$\,\{ f<s\}\,$ is not Stein. Since $\,G/K\,$ is biholomorphic to a Stein domain in $\,\C^n\,$ and 
$\,f\,$ is plurisubharmonic, this contradicts \cite{Car73}, Thm.\,B, p.\,419. Hence $\,\widehat f\,$ belongs to $\,ConvDec(\Omega)$, as claimed.

In order to  prove the converse,   as in the previous section, 
 for  $\,\varepsilon>0$   consider the convex $C$-invariant set 
$\,\Omega_\varepsilon:=\{{\bf  y} \in\Omega \ : \ d_\Omega({\bf  y})>\varepsilon \}\,$.    
For $\,\widehat f\,$ in $\,ConvDec(\Omega)\,$, let  $\,\widehat f_\varepsilon :\Omega_\varepsilon \to \R\,$ be the function 
$$\textstyle \widehat f_\varepsilon ({\bf  y}):=\int_{\R^r} \widehat f({\bf  y} +\varepsilon {\bf  w}) \widehat \sigma({\bf  w})d{\bf  w} +\varepsilon \widehat h\,,$$
where $\,\widehat h\,$ is the function given in (\ref{H-HAT}) and $\widehat \sigma: {\R^r} \to \R$ is a smooth, positive,  radial function  (only depending on $R^2=\Vert{\bf w}\Vert^2$), 
with support in
$\B_1({\bf 0})$,   such that $\widehat\sigma'(R^2)<0$ and  $\int_{\R^r} \widehat\sigma ({\bf  w})d{\bf  w} =1$.
Arguments analogous  to those used  in Lemma \ref{APPROSSIMANTI} show that the functions 
$\,\widehat f_\varepsilon\,$ are  in   $\,ConvDec^{\infty,+}(\Omega_\varepsilon)$. 
Then (i) implies that
the corresponding functions $\,f_\varepsilon\,$ belong to  $\,Psh^{\infty,+}(D)^N\,$ and consequently  
  $\,f\,$ belongs to 
$\,Psh\,(D)^N$.\pn
 (iv) follows directly from the definition of  $\,Psh^{+}(D)^N\,$ and of $\,ConvDec^+(\Omega).$ 

\smallskip
Finally, from the inclusions 
$$\begin{matrix}
ConvDec^+(\Omega) & \subset  & ConvDec(\Omega) &\subset &C^0(\Omega) \cr
    \cup  &   &  \cup &\cr
ConvDec^{{\infty},+}(\Omega)  & \subset  & ConvDec^{\infty}(\Omega)  & & 
\end{matrix}  
$$
it follows that  all the above functions on $\Omega$ are continuous,   and so are the corresponding  $\,N$-invariant plurisubharmonic functions on $D$. 
 \end{proof}

\medskip
\section{The Siegel domain point of view}

The goal of this section is to present an alternative characterization of Stein $N$-invariant  domains in an irreducible  Hermitian symmetric space $G/K$, realized as a Siegel domain. 

Denote by $S=NA$  the real split solvable group arising from the Iwasawa decomposition of $G$ subordinated to $\Sigma^+$. 
With the complex structure $J$ described in (\ref{COMPLEXJ}) and the linear form $f_0\in\s^*$ defined by  $f_0(X):=B(X,Z_0)$, where $Z_0\in Z(\k)$ is the element inducing the complex structure on~$\p$, 
 the Lie algebra $\s=\n\oplus \a$ of $S$ has the structure of a {\it normal $J$-algebra} (see \cite{GPSV68} and  \cite{RoVe73},  Sect.\,5,\,A).

This means in particular that $\omega(X,Y):=-f_0([X,Y])$ is a non-degenerate skew-symmetric bilinear form on $\s$ and 
that the symmetric bilinear form
 $ \langle X,Y\rangle:=-f_0([JX,Y])$ is the $J$-invariant positive definite inner product on $\s$  defined in~(\ref{INNERPRODUCT}). 

The adjoint  action of $\a$ on $\s$  decomposes $\s$ into the orthogonal direct sum of the restricted root spaces. 
Moreover, the adjoint action of the element $A_0={1\over 2}\sum_j A_j\in\a $  decomposes $\s$ and $\n$ as
  $$\s=\s_{0}\oplus  \s_{1/2}\oplus \s_1,\qquad \n_j=\n\cap\s_j$$ where
\begin{equation}\label{DOUBLEDECOMP} 
\s_0=\a\oplus\bigoplus_{1\le j<l\le r} \g^{e_j-e_l},\quad  \s_{1/2}=\oplus_{\atop 1\le j\le r}\g^{e_j}, \quad \s_1=\oplus_{\atop 1\le j\le r}\g^{2e_j}\oplus\bigoplus_{1\le j<l\le r}\g^{e_j+e_l}.\end{equation}
 
 \sn
 Let $E_0:= \sum E^j$. The  orbit 
 \begin{equation}\label{HOMOGENEOUSCONE} V:=Ad_{\exp\s_0}E_0  \end{equation} 
is a sharp  convex homogeneous selfadjoint cone in $\s_1$ and 
$$F\colon \s_{1/2}\times\s_{1/2}\to \s_1+i\s_1,\qquad F(W,W')={1\over 4}([JW',W]-i[W',W]),$$  is a  $ V$-valued Hermitian form, i.e.  it is sesquilinear and  $F(W,W)\in \overline V$, for all $W \in \s_{1/2} $.   
The  Hermitian  symmetric space  $G/K$ is realized as a Siegel domain 
 in $\s_1^\C\oplus\s_{1/2}$ as follows
 $$ D( V,F)=\{(Z,W)\in \s_1\oplus i\s_1\oplus \s_{1/2}~|~Im(Z)-F(W,W)\in V\}.
 $$  
 If $\s_{1/2}=\{0\}$ 
 then $G/K$ is of {\it tube type}, otherwise it is of {\it non-tube type}.
The group $S$ acts on  $ D( V,F)$  by the affine transformations 
\begin{equation}\label{ACTION} (Z,W)\mapsto (Ad_sZ+a+2iF(Ad_sW,b)+iF(b,b),Ad_sW+b),  \end{equation}
where $s\in\exp\s_0$,~$a\in \s_1$, and $b\in  \s_{1/2} .$ Recall that   $J\a=\oplus_j\g^{2e_j}$, (cf.\,(\ref{CPLXBIS})) and denote by $J\a^+$ the   positive octant in $J\a$.  One easily verifies that if $E\in J\a^+$, then $Ad_{\exp\a}E=J\a^+$. This and the fact that $S$ acts freely and  transitively on $D( V,F)$ imply that every $N$-orbit meets the set $J\a^+$ is a unique point.

Let $D $ be an $N$-invariant domain  in a symmetric Siegel domain. 
Then  $$D=\{(Z,W)\in D( V,F)~|~ Im(Z)-F(W,W)\in V_D\},    
$$
where  $ V_D$ is an $Ad_{\exp \n_0}$-invariant open  subset in $ V$, determined  by $$i V_D:=D\cap i V.$$
The   $r$-dimensional set  
$$\VD:= V_D\cap J\a^+ , $$  
  intersects every $N$-orbit of $D$ in a unique point, and it 
is the base of an $r$-dimensional  tube domain in $J\a\oplus iJ\a$. The map  $R\exp\a\cdot eK\to  R\exp\a\cdot (iE_0,0)$ 
$$\textstyle \exp(\sum_jx_jE^j) \exp({1\over 2}\sum_k\ln(y_k)A_k)K\mapsto (i Ad_{\exp({1\over 2}\sum_k\ln(y_k)A_k)}E_0+\sum_jx_jE^j ,0)$$
is the inverse of the map ${\mathcal L}$ of Proposition \ref{FACT1} (cf. Cor.\,\ref{ASSOCIATEDTUBE}).

Let $C$ be the cone defined in~(\ref{CONE}). Then the characterization of $N$-invariant Stein domains in a symmetric  Siegel domain   can be formulated as follows.

\begin{prop} \label{STEINDOMAINS} Let $D$ be an $N$-invariant domain in an irreducible symmetric Siegel domain. Then $D$ is Stein if and only if $\VD$ is convex and  $C$-invariant.
 \end{prop}

 \smallskip
In order to prove the above proposition, we need some preliminary results. For this we separate the tube and the non-tube case.

\sn
{\bf The tube case}.
 Denote by $conv( V_D)$ the convex hull of $ V_D$ in $\s_1$. Since  $ V_D$ is $Ad_{\exp\n_0}$-invariant and the action is linear, then also $conv( V_D)$ is $Ad_{\exp\n_0}$-invariant. Denote by $p\colon \s_1\to J\a  $ the   projection onto $J\a$, parallel to $\oplus \g^{e_j+e_l}$. Denote  by 
 \begin{equation}\label{EJDUAL}
 (E^1)^*, \ldots,(E^r)^* \end{equation}
  the elements in  the dual $\n^*$  of $\n$, with the property that $(E^j)^*(E^l)=\delta_{jl}$ and $(E^j)^*(X^\alpha)=0$, for all $X^\alpha\in \g^\alpha$, with $\alpha\in \Sigma^+\setminus\{2e_1,\ldots,2e_r\}$.

\sn
\begin{lem}  \label{PROPERTIES} 
One has
\begin{itemize}

\item [(i)]  Let $E=\sum x_kE^k \in J\a^+$, where $x_k\in \R^{>0}$. Then $$p(Ad_{\exp\n_0} E)=E+C_{r-1}.$$

In particular,   
  $(E^r)^*(Ad_{\exp tX} E)=x_r ,$    for all $X\in\n_0$ and $t\in\R$. 

\sn
  \item [(ii)] Let $X\in\g^{e_j-e_l}$.  Then $[[E^l,X],X]=s E^j$, for some $s\in\R^{> 0}$.

\sn
\item [(iii)] One has $p(conv( V_D))=conv(p( V_D))  $.
\end{itemize}
\end{lem}

\begin{proof}

\sn
(i)  Let $E\in J\a^+$ and let 
 $h_0\in\exp\n_0$, where  $\n_0=\oplus_{\atop 1\le i<j\le r}\g^{e_i-e_j}$. 
 By Theorem 4.10 in \cite{RoVe73}, for every $1\le i<j\le r$  there exists a basis $\{E_{ij}^p\}$ of $\g^{e_i-e_j}$, with coordinates $\{ x_{ij}^p\}_p$, such that    
$$\textstyle  (E^i)^*(Ad_{h_0}   E) =  x_i (1+\sum_{\atop p, \,j>i} (x_{ij}^p)^2)  $$
(formula (4.13) in \cite{RoVe73}). Since $i<r$, one has
$p( Ad_{\exp X}   E)= E+C_{r-1}$, as claimed. In particular    the $r^{th}$ coordinate of $E$ does not  vary  under the $Ad_{\exp\n_0}$-action.

 \sn
(ii)  Let $X\in\g^{e_j-e_l}$.  
Then $\exp tX\in\exp\n_0$ and the  curve 
 $$\textstyle Ad_{\exp tX}E_0 =\exp ad_{tX}(E_0)=E_0+t[X,E^l]+{t^2\over 2} [X,[X,E^l]], ~t\in\R,$$ 
 is   contained in $V$.    By Lemma\,\ref{NBRACKETS}\,(a), its projection onto $J\a$ is given by 
  $$\textstyle p(Ad_{\exp tX}E_0)= (E^j)^*(Ad_{\exp tX}E_0 )E^j= (1+\frac{t^2}{2}s)E^j ,  $$
  for some  $s\in \R$, $s\not=0$. Now (i)   implies that  $ 1+\frac{t^2}{2}s  >0$, for all $t\in\R$.  Therefore $s>0$, as claimed.

 \sn
(iii)  We prove the two inclusions. By the linearity of $p$, the set  $p(conv( V_D))$ is convex and contains  $p( V_D)$. 
Hence, $p(conv( V_D))\supset conv(p( V_D))$.
Conversely, let $z\in conv( V_D)$. Then  there exist $t_0\in(0,1)$ and $x,y\in V_D$ such that $z=t_0x+(1-t_0)y$. Since  $p(z)= t_0p(x)+(1-t_0)p(y)$, one has 
$ p( conv( V_D))\subset conv(p( V_D)) $. 
 \end{proof}

\bn
{\bf The non-tube case}.  
Denote  by $\, \widetilde p\colon \s_1^\C\oplus\s_{1/2} \to iJ\a\,$ the   projection  onto~$\, iJ\a$ parallel to $\s_1\oplus i(\oplus\g^{e_j+e_l})\oplus \s_{1/2}$.

\sn
\begin{lem} \label{PROPERTIES2} Let $E\in J\a^+$. Then $\widetilde p(N\cdot(iE,0))=i(E+\overline C_r)$.
\end{lem}

\begin{proof} 
  The $N$-orbit of the point $(iE,0)\in \s_1^\C\oplus \s_{1/2}$ is given by
 \begin{equation}\label{ORBITA2} N\cdot (iE,0)=S_{1/2}S_1Ad_{\exp\n_0}(iE,0)=(a+i(Ad_{\exp\n_0}E+F(b,b)),b),\end{equation} where $a\in\s_1$ and $b\in\s_{1/2}$.
By (\ref{ORBITA2}) and  Lemma \ref{PROPERTIES}\,(i), one has
 $\widetilde p(N\cdot (iE,0))=i(E+C_{r-1}+  \tilde p(F(\s_{1/2},\s_{1/2})))$.  
Since in  the symmetric  case   $ \{[Jb,b],~b\in\s_{1/2}\} =\overline{J\a^+}$,  
it follows  that  $\tilde p(N\cdot(iE,0))=i(E+\overline C_r)$, as claimed. 
\end{proof}

 \sn
\begin{remark} \label{EXPLANATION}
{\rm (a)}  Statement (i) in Lemma \ref{PROPERTIES} 
explains why in Prop.\ref{LEVI} (iii)  no conditions appear on $\frac{\partial \tilde f}{\partial a_r}$. 

\noindent 
{\rm (b)}  Statement (ii) in Lemma \ref{PROPERTIES} and  the fact that  $F(b,b)=[Jb,b]$, for $b\in\s_{1/2}$,  takes values in $\overline{J\a^+}$, 
explain  why the real constants $s$ and $t$ in Lemma \ref{NBRACKETS}(a)(b)  and later in Proposition \ref{LEVI}(iii)(iv) are strictly positive.
\end{remark}

\bn
{\it Proof of Proposition} \ref{STEINDOMAINS}.   {\bf The tube case}.   An $N$-invariant domain $D$ in a symmetric tube domain $D(V)$ is itself a tube domain with base the $Ad_{\exp\n_0}$-invariant set $ V_D$.  
Hence  all we have to prove is that  $ V_D$ is convex if and only if  $\VD$ is convex and  $\VD+C_{r-1}\subset \VD$. 
 
Assume that $ V_D$ is convex. Then $\VD$ is convex, being the intersection of $ V_D$ with the 
positive  octant $J\a^+$. To prove that $\VD$ is $C$-invariant,  let $E=\sum_jx_jE^j\in \VD$, where $x_j>0$,  and let 
 $X\in \g^{e_j-e_l} $ be a  non-zero element.   
For every $t\in\R$, 
$$\textstyle Ad_{\exp tX}E=E+tx_l [X,E^l] +{1\over 2}t^2 x_l[X,[X,E^l]] $$
 lies in $ V_D$ and, by the convexity assumption, so does $E+ {1\over 2}t^2x_l [X,[X,E^l]] =E+t^2s x_lE^j $,   where $s> 0$ (cf.\,Lemma  \ref{PROPERTIES} (ii)).  
This argument applied to all $j=1,\ldots,r-1$ and the convexity of $\VD$  show  that  $\VD+C_{r-1}\subset \VD$, as desired.

 Conversely,  assume  that $\VD$ convex and $C$-invariant.  We prove the convexity of $ V_D$ by showing  that  $conv( V_D)\subset  V_D$. 
From  Lemma \ref{PROPERTIES} (ii) and the $C$-invariance of $\VD$, one has 
$$\textstyle p( V_D)=p(Ad_{\exp\n_0}\VD)=\VD+C_{r-1} \subset \VD.$$
Moreover, from Lemma \ref{PROPERTIES} (iii), the above inclusion  and   the convexity   of $\VD$, one has
$$\textstyle conv( V_D)\cap J\a\subset p(conv( V_D))= conv(p( V_D))\subset \VD.
 $$
 Finally, from the $Ad_{\exp\n_0}$-invariance  of $conv( V_D)$ it follows that 
 $$\textstyle conv( V_D)=Ad_{\exp\n_0}  (conv( V_D)\cap J\a)\subset Ad_{\exp\n_0} \VD= V_D.$$
This completes the proof of the proposition in  the tube case.

 \mn
{\bf The non-tube case}.  Let $D$ be an $N$-invariant domain in a Siegel domain $D( V,F)$.  Denote by $conv(D)$ the convex hull of $D$ in  $\s_1^\C\oplus\s_{1/2}$. As $N$ acts on $D$ by affine transformations, also  $conv(D)$ is $N$-invariant.  

If $D$ is Stein, then $D\cap \{W=0\} $ is a Stein tube domain in $\s_1^\C$ with base $ V_D$. By  the result  for the tube case and    Lemma \ref{PROPERTIES2}, 
 $\VD$ is convex and~$\VD+\overline C_r\subset \VD$. 

Conversely, assume that  $\VD$ is convex and $C$-invariant, i.e. $\VD+\overline C_r\subset \VD$ (see Def.\,\ref{CONEINVARIANT}).  We are going to prove 
 that $D$ is convex. 
By Lemma \ref{PROPERTIES2}, one has
$$\widetilde p(D)=\widetilde p(N \cdot \VD)=i(\VD+\overline C_r)\subset i\VD. $$
Moreover, 
$$conv(D)\cap iJ\a\subset \widetilde p(conv(D))=conv(\widetilde p(D))\subset i\VD.$$
By the $N$-invariance of $conv(D)$,  one obtains
$$conv(D)=N\cdot (conv(D)\cap i J\a)\subset N\cdot i\VD=D.$$
Hence $D$ is convex and therefore Stein (cf. \cite{Gun90}, Vol.1, Thm.10, p. 67). This concludes the proof  of the proposition.
\qed

\mn
\begin{rem} The assumption $\VD+C_{r}\subset \VD$ implies $\VD+C_{r-1}\subset \VD$ and in particular $ V_D$ is convex. This means that if $D\subset D( V,F)$ is Stein, then the tube domain $D \cap \{W=0\}$ is Stein. 
The converse may not hold true, as $ V_D=Ad_{\exp\n_0}\VD$ convex does not imply $\VD+C_r\subset \VD$. 
\end{rem}

 \bigskip
\section{Appendix: $\,N$-invariant potentials for the Killing metric.}
\label{OMEGAN}


\medskip

Let $G/K$ be a non-compact, irreducible  Hermitian symmetric space. The
Killing form $B$ of $\g$, restricted to $\p$,  induces  a $G$-invariant
K\"ahler metric on $G/K$, which we refered to as the Killing metric. 
In this section  we exhibit an $\,N$-invariant
potential of the Killing metric and the associated moment map in a Lie theoretical fashion. 
All the $\,N$-invariant
potentials of the Killing metric  are detemined in Remark \ref{UNIQUE}.

 Let $f\colon G/K\to\R$ be  a smooth  $N$-invariant function.   The   map $\mu \colon G/K\to\n^*$,   defined  by   \begin{equation}\label{MOMENT2} \,\mu_f(z)(X):=d^cf(\widetilde X_{z}),\end{equation} for $X\in\n$, is  $N$-equivariant (cf. (\ref{DCf})). 
If  $f$ is strictly plurisubharmonic,  then it is  referred to as the moment map associated with $f$.

\bn

\bigskip
\begin{prop}
\label{POTENTIALN} 
Let $\,z=naK \in G/K$, where  
$\,n \in  N$, $\,a=\exp H \in A\,$
and  $\,H=\sum_ja_jA_j\in\a$. Let  ${\bf b}$ be the constant defined in (\ref{BI}).

\sn
\item{$(i)$} The $\,N$-invariant function $\,\rho:G/K \to \R\,$ defined by 
$$\, \textstyle \rho(na K):=  \textstyle -\frac{1}{2}{\sum_{j=1}^r}B(H,\,A_j)=  \textstyle -\frac{\bf b}{2}(a_1+ \dots+a_r)\,,$$
is a potential of  the Killing metric. 

\sn
\item{$(ii)$} The moment map $\,\mu_\rho: G/K \to \n^*\,$ associated with
$\,\rho\,$ is given by 
\begin{equation}\label{MOMENTIDENTITY}  \,\mu_\rho(naK)(X) =\textstyle -\frac{\bf b}{4}\sum_{j=1}^r e^{-2a_j}(E^j)^* (\Ad_{n^{-1}} X)=
 B(Ad_{n^{-1}}X , Ad_a Z_0) 
\,,\end{equation}
where  $\,X\in\n$, and  the $\,(E^j)^*\,$ are defined in (\ref{EJDUAL}).
\end{prop}

\mn
\begin{proof}
(i)  Let $naK\in G/K$, where $a=\exp H$  and $H=\sum_ja_jA_j$.  
The function  $\widetilde  \rho:\a \to \R\,$  associated to $\rho$ is given by
 $\textstyle \widetilde\rho(H)= -{1\over 2} {\sum_{j=1}^r}a_jB(A_j,\,A_j)\, $ 
(cf.\,(\ref{EFFETILDE})). 
In order to obtain (i), we first prove the identities  (\ref{MOMENTIDENTITY}).
By (\ref{MOMENT2}) and  (\ref{DCf2}), one has  
\begin{equation}\label{MOMENT3} \textstyle \mu_\rho( aK)(X)=  d^c\rho( \widetilde X_{aK}) =-\frac{\bf b}{4}\sum_{j=1}^r e^{-2a_j}(E^j)^* ( X).\end{equation}
By (\ref{INNERPRODUCT}),   one has
$$\textstyle (E^j)^*(X) = B(X, \theta E^j)/B(E^j, \theta E^j) = 2B(X,\frac{1}{2}(E^j+\theta E^j) )/B(E^j, \theta E^j).$$
Since 
$$ \textstyle {\bf b}:= B(A_j,A_j) = B (I_0A_j,I_0A_j) = B(E^j - \theta E^j,E^j - \theta E^j)= - 2B(E^j, \theta E^j) $$
and $Z_0=S_0+{1\over 2}\sum_jE^j+\theta E^j $, for some $S_0\in\m$ (cf.\cite{GeIa21}, Sect.\,2),
one obtains
$$ \textstyle -\frac{\bf b}{4}\sum_{j=1}^r e^{-2a_j}(E^j)^*(X) = -\frac{\bf b}{2} \sum_{j=1}^re^{-2a_j}  B(X, \frac{1}{2}(E^j+\theta E^j)/B(E^j, \theta E^j) $$
$$= \textstyle
\sum_{j=1}^r B(X, Ad_a \frac{1}{2}(E^j+\theta E^j)) = B(X,  Ad_a Z_0)\,,$$
and 
 (\ref{MOMENTIDENTITY}) follows from the $N$-equivariance   of $\mu_\rho$.

\medskip
\noindent
Next we are going to show that on $\p\times\p$ one has 
\begin{equation}\label{BELLA} h_\rho(\,a_*\cdot\,,\,a_*\cdot\,)=B(\,\cdot\,,\,\cdot\,).\end{equation} 
 Every $X\in\s$ decomposes as $X=(X-\phi(X))+\phi(X)\in \k\oplus \p$ (see Sect.\,2).  
 Since the projection $\,\phi:\s \to \p\,$ is a linear isomorphism, (\ref{BELLA})  is equivalent to  
\begin{equation}
\label{TESI}
\textstyle \,h_\rho(a_*X,a_*Y)=h_\rho(a_* \phi(X),a_*\phi(Y))=B\big ( \phi(X),\,\phi(Y) \big)
= -\frac{1}{2}B\big ( X,\,\theta Y  \big)\,,  \end{equation}
  for all $\,X$, $\,Y\,$ in $\,\s\,$.    
   By Proposition \ref{LEVI}(i),  it is sufficient to consider $X,\,Y$ both in the same block $a_*\a$, $a_*\g^{e_j-e_l}$, and $a_*\g^{2e_j}$.

  Let $A_j, A_l \in\a$, be as in (\ref{NORMALIZ1}). Then, by
 (ii) of Proposition \ref{LEVI}, one has 
  $$h_\rho(a_*A_j,a_*A_l)=\delta_{jl}B(A_l,\,A_l)=B(A_j,A_l)\,.$$

\mn

Let $X,\,Y\in\g^\alpha$, with $\alpha= e_j-e_l$ or $\alpha= e_j$. Then $JY\in\g^\beta$, for  $\beta=e_j+e_l$ or $\beta=e_j$, respectively.  From 
(\ref{FORMULONE}) and (i) one  obtains  
$$\textstyle h_\rho(a_*X,a_*Y)=-e^{\alpha(H)+\beta(H)}d^c\rho(\widetilde {[JY,X ]}_z)$$
 \begin{equation}\label{HPQ} \textstyle = - e^{\alpha(H)+\beta(H)}B([JY,X],Ad_aZ_0).\end{equation}

\REM{
$I_0\p[\alpha]=\p[\beta]$, with $\beta \in\Sigma^+$.
Let   $\,P,\,Q \in \p[\alpha]\,$ and write
$\,P=X -\theta X \in\p[\alpha],$ for some $X\in\g^\alpha$, and $\,I_0Q =Y-\theta
Y \in \p[\beta]$, for some $Y\in \g^\beta$. By (\ref{FORMULONE}) one has 
  $$h_\rho(a_*P,a_*Q)=-4 e^{\alpha(H)} e^{ \beta(H)}
d^c \rho( \widetilde {[Y,X ]}_z)\,.$$
Since $\,d^c\rho(\widetilde {[Y,X ]}_z)=\mu(z)([Y,X ])$, by (ii) one has  
\begin{equation}\label{HPQ} h_\rho(a_*P,a_*Q)=\textstyle
-2e^{\alpha(H)+ \beta(H)}\sum_j e^{-2a_j}B([Y,X], I_0A_j).\end{equation}      }

\sn
From the invariance properties of the
Killing form $\,B\,$,  the decomposition of $X$ and $JY$ in $\k\oplus \p$ and the identity $\phi(J\cdot) =I_0\phi(\cdot)$ (cf. (\ref{COMPLEXJ})),   one has
$$\textstyle B([JY,X],Ad_aZ_0)=B(Ad_{a^{-1}}[JY,X],Z_0)=e^{-(\alpha(H)+\beta(H))}B([JY,X],Z_0) $$
$$\textstyle =e^{-(\alpha(H)+\beta(H))}\left(B([JY-\phi(JY),X-\phi(X)],Z_0)+B([\phi(JY),\phi(X)],Z_0)\right)$$
$$\textstyle =e^{-(\alpha(H)+\beta(H))}B([Z_0,\phi(Y)],\phi(X)],Z_0)=e^{-(\alpha(H)+\beta(H))}B(\phi(X),[Z_0,[Z_0,\phi(Y)]])$$
$$\textstyle=-e^{-(\alpha(H)+\beta(H))}B(\phi(X),\phi (Y))=\frac{1}{2}e^{-(\alpha(H)+\beta(H))}B(X,\theta Y).$$
It follows  that \begin{equation} \label{KILLINGPOTENTIAL}
\textstyle \,h_\rho(a_* X,a_*Y)=-\frac{1}{2}B\big (X,\,\theta Y \big),\end{equation}
as desired. This concludes the proof of (i).

\sn
(ii) The identity (\ref{KILLINGPOTENTIAL}) implies  that the $N$-invariant function $\rho$ is strictly 
plurisubharmonic. Hence
$\,\mu_\rho\,$ is the moment map associated to $\rho$.
\end{proof}

\medskip
 \begin{remark} 
\label{CONSTANTS} Combining (\ref{ESSE}) and (\ref{TI}) in Proposition \ref{LEVI}  with  (\ref{TESI}), we obtain the exact value of the   positive quantities $s$ and $t$ 
$$\textstyle s=\frac{4\Vert X\Vert^2}{{\bf b}}, ~\hbox{ for $X\in\g^{e_j-e_l}$}, \quad \hbox{and}\quad t=\frac{4\Vert X\Vert^2}{{\bf b}}, ~\hbox{ for $X\in\g^{2e_j}$}.$$

 \end{remark}

\medskip
 \begin{remark} 
\label{RESTRICTION} The map $\,\mu_G:G/K \to \g^*\,$  given by  $\mu_G(gK)(\cdot ):=B(\Ad_{g^{-1}}\,\cdot\,,Z_0)$   is a
 moment map  for the  $\,G$-action on $\,G/K$. 
The moment map $\mu_\rho$  in (ii) of    
Proposition \ref{POTENTIALN}  can be obtained  by   restricting   $\mu_G(naK)$ to 
$\,\n$. Namely, for $\,X \in \n\,$  and $\,naK \in G/K\,$ one has 
  $$\,\mu_\rho(naK)(X)= \mu_G(naK)(X)=B(\Ad_{(na)^{-1}}\,X\,,Z_0).$$ \end{remark} 
 
 \bigskip
In the next remark, all possible $N$-invariant potentials of the Killing metric are determined.
 
\begin{remark} 
\label{UNIQUE} Let $\rho\colon G/K\to\R$ be the potential of the Killing metric given in Proposition \ref{POTENTIALN} and let $\,\sigma\,$ be another $\,N$-invariant potential.  Let $\widehat \rho$ and $ \,\widehat \sigma$ be the corresponding functions on $ (\R^{>0})^r$ defined in (\ref{EFFEHAT}).
\sn
\pn 
(a)  In the non-tube case, one has $\widehat \sigma=\widehat \rho +d$, and therefore $\sigma =\rho +d$, for  some~$d\in\R$;   
\sn
\pn
(b) In the tube case,   one has $\widehat\sigma({\bf  y})=\widehat\rho({\bf  y})+cy_r+d$, for $c,d\in \R$. In particular 
$$\sigma(n\exp (L({\bf y}))K) =\rho(n\exp (L({\bf y}))K)+  c y_r+d,$$
where  $\,n\in N\,$,  $\,{\bf y} =(y_1,\ldots,y_r)\in (\R^{>0})^r$,  and  $c,d \in\R$. 
\smallskip
\end{remark}

\begin{proof} Let   $f:=\sigma -\rho$ be the difference of the two potentials. Then $f$  is a smooth $\,N$-invariant function  on $G/K$ such that $dd^cf(\cdot,J \cdot) \equiv 0$.  Let $\widehat f\colon\Omega\to\R$ be the associated function.
\pn
(a)  In the non-tube case, by  Proposition \ref{LEVI}\,(iv)   and  (\ref{DEREFFETILDE}),  the   function $\widehat f$ 
satisfies~$\frac{\partial\widehat f}{\partial y_j}\equiv 0$, for all $j=1,\ldots r $. Hence  $\widehat f$ is constant on $\,(\R^{>0})^r\,$ and
  $\,f\,$ is constant on~$G/K$. 
\pn
(b) In the tube case, from Proposition \ref{LEVI}, (\ref{HESSEFFETILDE}) and (\ref{DEREFFETILDE}),  it follows that~$\frac{\partial\widehat f}{\partial y_j}\equiv0$, for all $j=1,\ldots r-1$, and $\,
\frac{\partial^2  \widehat f}{\partial  y_r^2}  \equiv0$ . 
Hence   $\,\widehat f\,$  is an affine function of the variable $\,y_r\,$. Equivalently,  $\widehat\sigma({\bf y})=\widehat\rho({\bf y})+cy_r+d$, for $c,d\in\R$, as claimed.  \end{proof}

\mn
\pn
\begin{remark} 
\label{UNIQUE}  Let $D(V,F)$ be a symmetric Siegel domain. Then the Bergman kernel function $K(z,z)$ is $N$-invariant  and   $\ln K(z,z)$   is a potential of the Bergman metric. As both the Killing and the Bergman metric  are $G$-invariant,  they differ by a multiplicative constant. It follows that   $ \ln K(z,z)$ is a multiple of one of the $N$-invariant potentials of the Killing metric described in the above remark.
\end{remark} 

\medskip
\begin{exa}  
\label{DISC}
As an application  of Remark \ref{UNIQUE}, we compute all $N$-invariant potentials of the Killing metric for the upper half-plane in $\C$ and for the Siegel upper half-plane of rank 2.

\sn
(a) Let $\,G=SL(2,\R)\,$ and let $\,G/K\,$ be the  corresponding 
Hermitian symmetric space. Fix an Iwasawa decomposition $NAK$ of $G$.
Since  ${\bf b}=8$ and $r=1$, then the potential  of the Killing metric given in  Proposition\,\ref{POTENTIALN} is
$$\rho(naK)=-4a_1 \quad {\rm and} \quad
\textstyle \widehat \rho(y_1)=\rho(\exp L(y_1)K)=\ln \frac{1}{y_1^2}. $$

Realize $\,G/K\,$ as  the upper half-plane 
$\,\HH=\{z\in\C~|~Im(z)>0\}$, i.e. the orbit of $i\in\C$ under the $SL(2,\R)$-action by linear fractional transformations.
Fix 
$$N= \left \{ \begin{pmatrix}
1  & m  \cr
 0 & 1
\end{pmatrix} \ : \ m \in \R \,\right\} \, \quad\hbox{and}\quad A= \left \{ \begin{pmatrix}
e^{a_1}  & 0  \cr
 0 & e^{-a_1}
\end{pmatrix} \ : \ a_1 \in \R \,\right\}, \,$$
and let $\,\{x_1+iy_1 \in \C \ : \ y_1>0\}$ be tube associated to $\,G/K$.
Since   $$\,\textstyle x_1+iy_1 \to\exp({x_1E^1})\exp (\frac{1}{2} \ln y_1A_1)\cdot i=x_1+iy_1\, $$
 (cf. Prop.\,\ref{FACT1}), then   the potential $\rho$ on $\,\HH\,$ reads as $\,\rho(z) = \ln \frac{1}{({\rm Im} z)^2}. $

  If $\sigma \colon \HH\to \R$ is an arbitrary $N$-invariant potential of the Killing metric, then by Remark \ref{UNIQUE} 
$$\textstyle \sigma(z)= \ln \frac{1}{({\rm Im} z)^2}+c{\rm Im} z+d,\qquad c,d\in\R.$$

\sn
(b) The Siegel upper half-plane of rank 2
$$\,\mathcal P=\{W=S+iT\in M(2,2,\C)~|~{}^tW =W,~T>0\}, $$
of   $2\times 2$ complex symmetric matrices with positive definite imaginary part, is the orbit of $\,iI_2\,$ under the action  by linear fractional transformations of the real symplectic group $Sp(2,\R)$.
Fix the  Iwasawa decomposition such that 
$$N=\left\{  \begin{pmatrix} {\bf n}&{\bf m}\\ {\bf 0}&{}^t{\bf n}^{-1}\end{pmatrix}\right\} ,\qquad A=\left\{  \begin{pmatrix} {\bf a}&{\bf 0}\\ {\bf 0}&{} {\bf a}^{-1}\end{pmatrix}\right\} , $$
where ${\bf n}$ is unipotent,  ${\bf n}\,{}^t{\bf m}$  is symmetric and  $\,{\bf a}=\begin{pmatrix} e^{a_1}&0\\ 0& e^{a_2}\end{pmatrix}$, with 
$a_1$, $a_1$  coordinates  in $\a$  with respect to the basis defined in Lemma \ref{BASIS2}.

As ${\bf b}=12$, the  potential  of the Killing metric defined in Proposition\,\ref{POTENTIALN}
 is given by $$\textstyle \rho(naK)= -6(a_1+a_2)\,\quad \hbox{and}\quad 
 \textstyle\widehat\rho(y_1,y_2)=\rho(\exp L(y_1,y_2)K)=\ln\frac{1}{(y_1y_2)^3}.$$
 A matrix $\,S+iT \in{\mathcal P}\,$ can be expressed in a unique way as 
$$\,na\cdot iI_2=n\cdot  \begin{pmatrix} ie^{2a_1}&0\\ 0&ie^{2a_2}\end{pmatrix}.$$
If  $\,T=\begin{pmatrix} t_1&t_3\\ t_3&t_2\end{pmatrix}$, a simple computation shows that 
 $\,e^{2a_1}=t_1-t_3^2/t_2\,$ and $\,e^{2a_2}=t_2$. 
Hence $y_1=t_1-t_3^2/t_2$, $y_2=t_2$ and $\,\rho(S+iT)=\ln\frac{1}{(t_1t_2-t_3^2)^3}.$ 

 If $\sigma$ is an arbitrary $N$-invariant potential of the Killing form, then by Remark\,\ref{UNIQUE}
 $$\textstyle\sigma(S+iT)=\ln\frac{1}{(t_1t_2-t_3^2)^3} +c t_2+d,\qquad {\rm for \ some \ }c,d\in\R.$$
\end{exa}



\bigskip
\begin{thebibliography}{CoLo86}
\medskip

 \medskip 
 \bibitem[Car73] {Car73}
 {\sc Carmignani R.} 
  {\it Envelopes of holomorphy and holomorphic convexity.}
  Trans. of the AMS {\bf 179} (1973) 415-431.

\medskip
\bibitem[CoLo86]{CoLo86}
{\sc Coeur\'e~G., Loeb~J.-J.}
{\it  Univalence de certaines enveloppes d'holomorphie.} C.\,R. Acad. Sci. Paris S\'er. I Math. {\bf 302}  (1986) 59--61.

\medskip 
\bibitem[Fle78] {Fle78}
 {\sc Flensted-Jensen M.} 
 {\it Spherical functions of real semisimple Lie groups. A method of reduction to the complex case.}
  J. Funct. Anal.  (1) {\bf 30} (1978) 106--46.


\medskip
\bibitem[GeIa21]{GeIa21}
{\sc Geatti~L., Iannuzzi~A.}
{\it  Invariant plurisubharmonic functions on non-compact Hermitian symmetric spaces.} Math. Zeit. {\bf 300},  1 (2021) 57--80.

\medskip
\bibitem[GPSV68]{GPSV68}
{\sc Gindikin S., Pyatetskii-Shapiro I., Vinberg E.}
{\it In Geometry of bounded domains}. CIME 1968, Ed. Cremonese, Roma 1968, \,3--87. 

\smallskip
\bibitem[Gun90]{Gun90}
{\sc Gunning R. C.}
Introduction to Holomorphic Functions of Several Variables, Vol I:
Function Theory.
Wadsworth \& Brooks/Cole, 1990.

\smallskip
\bibitem[HeSc07]{HeSc07}
{\sc Heinzner, P., Schwarz G.\,W.} 
{\it Cartan decomposition of the moment map.}
Math. Ann. {\bf 337} (2007)  197--232.

\smallskip
\bibitem[H\"or94]{Hor94}
{\sc H\"ormander~L.}
Notions of convexity.
Birkh\"auser, Basel–Boston–Berlin, 1994.


\smallskip
\bibitem[Ran86]{Ran86}
{\sc  Range~R. M.} 
Holomorphic Functions and Integral Representations in Several Complex Variables.
GTM Vol. 108, Springer-Verlag, New York, 1986.


\medskip 
\bibitem[Ros63]{Ros63}
 {\sc Rossi H.} 
 {\it On envelopes of holomorphy.}
  Comm. Pure Appl. Math. {\bf 16} (1963) 9--17.


\medskip
\bibitem[RoVe73]{RoVe73}
{\sc Rossi~H., Vergne~M.}
{\it  Representations of Certain Solvable Lie Groups On Hilbert Spaces of Holomorphic Functions and the Application to the Holomorphic Discrete Series of a Semisimple Lie Group.} 
J. Funct. Anal. {\bf 13} (1973) 324--389.


 
\medskip
\bibitem[Wol72]{Wol72}
{\sc Wolf J.A.}
{\it Fine structure of Hermitian symmetric spaces.}
in Boothby, W., Weiss, G. Eds., Symmetric spaces.  Short Courses, Washington University, St. Luis (MO), 1969-1970,  Pure and App. Math. Vol.\,8, Dekker, New York, 1972, pp. 271-357.

\end{thebibliography}
\end{document}